\documentclass[a4paper,12pt]{amsart}
\usepackage{comment,amsmath,amssymb,amsthm,changepage,pifont,
            graphicx,tikz,soul,xcolor,longtable,mathtools, longtable}
\usepackage[new]{old-arrows}
\usepackage[english]{babel}
\usepackage[foot]{amsaddr}
\usetikzlibrary{patterns}
\usetikzlibrary{decorations.pathreplacing}
\usepackage{ytableau}
\allowdisplaybreaks

\newtheorem{theorem}{Theorem}

\newtheorem{lemma}[theorem]{Lemma}
\newtheorem{corollary}[theorem]{Corollary}
\newtheorem{proposition}[theorem]{Proposition}

\newtheorem{addendum}[theorem]{Addendum}
\theoremstyle{remark}
\newtheorem{remark}[theorem]{Remark}

\newcommand{\FF}{\mathbf{F}}

\newcommand{\ZZ}{\mathbf{Z}}
\newcommand{\BB}{\mathcal{B}}

\newcommand{\PP}{\mathbf{P}}

\newcommand{\OO}{\mathcal{O}}

\newcommand{\EE}{\mathcal{E}}

\newcommand{\II}{\mathcal{I}}
\newcommand{\Rco}{R_{\mathrm{co}}}

\DeclareMathOperator{\res}{R}
\DeclareMathOperator{\charac}{char}

\DeclareMathOperator{\Spec}{Spec}

\DeclareMathOperator{\id}{id}

\DeclareMathOperator{\Norm}{Norm}
\DeclareMathOperator{\AGL}{AGL}
\DeclareMathOperator{\Tr}{Tr}
\DeclareMathOperator{\disc}{disc}
\DeclareMathOperator{\Gal}{Gal}
\DeclareMathOperator{\mult}{mult}
\DeclareMathOperator{\Ind}{Ind}
\DeclareMathOperator{\Res}{Res}
\DeclareMathOperator{\Hom}{Hom}
\DeclareMathOperator{\Sym}{Sym}
\DeclareMathOperator{\sgn}{sgn}

\DeclareMathOperator{\vol}{vol}

\DeclareMathOperator{\ST}{ST}
\DeclareMathOperator{\ord}{ord}
\DeclareMathOperator{\hto}{ht}
\DeclareMathOperator{\ch}{ch}
\DeclareMathOperator{\chw}{chw}
\DeclareMathOperator{\lcm}{lcm}
\newcommand{\alg}{\mathrm{alg}}
\newcommand{\sep}{\mathrm{sep}}

\usepackage[colorinlistoftodos]{todonotes}
\usepackage[colorlinks=true, allcolors=blue]{hyperref}

\title{Scrollar invariants, syzygies and representations of the 
symmetric group II}

\author{Floris Vermeulen}
\address{Department of Mathematics, KU Leuven, Belgium}
\email{floris.vermeulen@kuleuven.be}

\date{}

\begin{document}


\maketitle

\begin{abstract}
Let $\varphi: C\to \PP^1$ be a degree $d$ cover of curves. In work by Castryck, Zhao and the author, we showed how one can attach to each partition $\lambda$ of $d$ a multi-set of scrollar invariants of $\lambda$ with respect to $\varphi$. We studied these invariants when $\varphi$ is simply branched, and related these new scrollar invariants to known geometric data. In this article we show how one can remove this simple branching condition, using the notion of $S_d$-closure as developed by Bhargava and Satriano. With this new framework, we are able to generalize all results from this work to arbitrary covers $C\to \PP^1$. In particular, we obtain a syzygy-free interpretation for the splitting types of the syzygy bundles in the Casnati--Ekedahl resolution of an arbitrary cover $\varphi: C\to \PP^1$. By using the Maroni bound, we are able to give new general bounds on these splitting types.

%The main new tools are the notion of $S_d$-closure as developed by Bhargava and Satriano, and the theory of higher Specht polynomials.
\end{abstract}

\section{Introduction}\label{sec:intro}

This article is an extension of recent work by Castryck, Zhao and the author~\cite{syzrep}, in which we studied scrollar invariants of curves using Galois theoretic tools. Throughout this work, all curves will be projective and reduced, but they may be singular or reducible. We use the words smooth and non-singular interchangeably. By a \emph{nice curve}, we will mean a smooth, projective and geometrically integral curve. 

Let $d$ be a positive integer. We will work over a field $k$ of characteristic zero or larger than $d$. Let $\varphi: C\to \PP^1$ be a separable degree $d$ cover of curves, where $C$ is a nice curve of genus $g$. Then the \emph{scrollar invariants}, see e.g.\ \cite{coppenskeemmartens}, of $C\to \PP^1$ are defined as the unique integers $e_1\leq \ldots \leq e_{d-1}$ for which
\begin{equation}\label{eq:splitting.varphi*OC}
\varphi_* \OO_C \cong \OO_{\PP^1} \oplus \OO_{\PP^1}(-e_1) \oplus \ldots \oplus \OO_{\PP^1}(-e_{d-1}).
\end{equation}
Equivalently, $e_i$ is the smallest integer $n$ such that $h^0(C, nD) - h^0(C, (n-1)D)>i$, where $D$ is a fibre of $\varphi$. Using Riemann--Roch, one can prove that 
\begin{enumerate}
\item $e_1 + \ldots + e_{d-1} = g+d-1$, and
\item $e_{d-1} \leq \frac{2g+2d-2}{d}$.
\end{enumerate}
This last inequality is known as the \emph{Maroni bound}. 

In~\cite{syzrep}, we developed a representation theoretic viewpoint for these scrollar invariants using Galois theory. We showed how one can attach to every irreducible representation of the symmetric group $S_d$ a multi-set of \emph{generalized scrollar invariants} of $\varphi: C\to \PP^1$. Recall that the irreducible representations of $S_d$ correspond naturally to partitions $\lambda$ of $d$. The corresponding irreducible representation is called the \emph{Specht module $V_\lambda$}. Then in more detail, to each partition $\lambda\vdash d$ one may associate a multiset $e_{\lambda, 1}\leq \ldots \leq e_{\lambda, \dim V_{\lambda}}$ of \emph{scrollar invariants corresponding to $\lambda$}. For the partition $\lambda = (d-1,1)$ one recovers the usual scrollar invariants $e_i$, while for $\lambda = (d)$ one obtains $0$. 

In~\cite{syzrep} we studied these generalized scrollar invariants with the added restriction that the map $\varphi$ is simply branched, i.e.\ all ramification patterns of $\varphi$ are of the form $(2, 1^{d-2})$. The main goal of this article is to extend the results from~\cite{syzrep} to arbitrary degree $d$ covers $\varphi: C\to \PP^1$ (in characteristic zero or larger than $d$). Several new ideas are needed for this, which we discuss in more detail in section~\ref{sec:differences}.

\subsection{Main results} Let us discuss the main results in some more detail. As above, let $\varphi: C\to \PP^1$ be a degree $d\geq 4$ cover of curves, where $C$ is a nice curve of genus $g$. Let $K/k(t)$ be the corresponding function field extension, where $k(t) = k(\PP^1)$ is the function field of $\PP^1$. Let $L/k(t)$ be the \emph{$S_d$-closure} of $K/k(t)$, as defined in~\cite{bhargavasatriano}. This is a formal construction of an étale algebra over $k(t)$ which behaves like the Galois closure of $K/k(t)$, but always with Galois group $S_d$. In case that the Galois closure $M/k(t)$ of $K/k(t)$ already has Galois group $S_d$, we simply have that $M=L$, but in general $L$ may be reducible. As an $S_d$-representation, $L$ is isomorphic to the regular representation of $S_d$, hence it decomposes as
\[
L\cong \bigoplus_{\lambda\vdash d} W_\lambda,
\]
where $W_\lambda \cong V_\lambda^{\dim V_{\lambda}}$ is the \emph{isotypic component of $\lambda$ in $L$}, i.e.\ it consists of all elements of $L$ which generate $V_\lambda$ as a representation (together with zero). We show how one may attach to each such $W_\lambda$ (equivalently each $\lambda$) a multi-set of $\dim V_\lambda$ integers called the \emph{scrollar invariants of $\lambda$ (with respect to $K$)}: 
\[
e_{\lambda_1}, e_{\lambda, 1}, \ldots, e_{\lambda, 1}, \quad e_{\lambda, 2}, \ldots, e_{\lambda, 2}, \quad \ldots \quad e_{\lambda, \dim V_\lambda}, \ldots, e_{\lambda, \dim V_{\lambda}} 
\]
where each $e_{\lambda, i}$ occurs with multiplicity $\dim V_\lambda$. For example, the unique scrollar invariant of the trivial partition $(d)$ is $0$, while those of the partition $(d-1,1)$ are the usual scrollar invariants $e_1, \ldots, e_{d-1}$ of $\varphi$. 

Our main results include extensions of all of the results in \cite{syzrep} to arbitrary covers $\varphi: C\to \PP^1$, thereby removing the condition of simple branching. These results include the following:
\begin{enumerate}
\item Proposition~\ref{prop:hooks}: a description of the scrollar invariants of hooks $(d-i, 1^i)$, generalizing~\cite[Prop.\,1]{syzrep},
\item Proposition~\ref{prop:duality}: a duality statement relating the scrollar invariants of a partition $\lambda$ with those of its dual $\lambda^*$, as in~\cite[Prop.\,31]{syzrep}, and
\item Corollary~\ref{cor:volume.formula}: a \emph{volume formula} for the sum of the scrollar invariants corresponding to $\lambda$, generalizing~\cite[Prop.\,30]{syzrep}.
\end{enumerate}
Moreover, our interpretation of the scrollar invariants of the partitions $\lambda_i = (d-i, 2, 1^{d-2-i})$ ($2\leq i\leq d-2$) as the splitting types of the syzygy modules of the Casnati--Ekedahl resolution of $\varphi$ extends as well. In particular, this gives a syzygy-free way of defining these splitting types, and gives a natural candidate definition in the number field setting, following the analogy from~\cite[Sec.\,7]{hessRR}.

In more detail, define 
\[
\EE = \OO_{\PP^1}(e_1)\oplus \OO_{\PP^1}(e_2)\oplus \ldots \oplus \OO_{\PP^1}(e_{d-1}).
\]
Then the morphism $\varphi: C\to \PP^1$ decomposes as $C\xhookrightarrow{\iota} \PP(\EE)\xrightarrow{\pi} \PP^1$, where $\pi$ is the natural bundle map and $\iota$ is the relative canonical embedding. Let $j: \PP(\EE)\to  \PP^{e_1 + \ldots + e_{d-1} +d-2} = \PP^{g+2d-3}$ be the \emph{tautological morphism}. Let $R = [\pi^*\OO_{\PP^1}(1)]$ be the \emph{ruling class} in the Picard group of $\PP(\EE)$ and let $H = [j^* \OO_{\PP^{g + 2d-3}}(1)]$ be the \emph{hyperplane class}. The Picard group of $\PP(\EE)$ is generated by these two elements. Then the \emph{Casnati--Ekedahl resolution of $\varphi$}~\cite{casnati_ekedahl} is the minimal free resolution of $C$ inside $\PP(\EE)$, which takes the form
\begin{multline}  
 0 \rightarrow \mathcal{O}_{\PP(\mathcal{E})} (-dH + (g-d-1)R) \rightarrow 
 \\
\bigoplus_{j=1}^{\beta_{d-3}} \mathcal{O}_{\PP(\mathcal{E})}(-(d-2)H + b_j^{(d-3)}R) \rightarrow \bigoplus_{j=1}^{\beta_{d-4}} \mathcal{O}_{\PP(\mathcal{E})}(-(d-3)H + b_j^{(d-4)}R)  \rightarrow   \label{Schresolution} \\
\cdots \rightarrow \bigoplus_{j=1}^{\beta_1} \mathcal{O}_{\PP(\mathcal{E})}(-2H + b_j^{(1)}R) \rightarrow \mathcal{O}_{\PP(\mathcal{E})} \rightarrow \mathcal{O}_C \rightarrow 0.
\end{multline}
The $b_{j}^{(i)}$ are certain integers which we call the \emph{splitting types} (see also Schreyer~\cite{schreyer}). Here 
\[
\beta_i = \frac{d}{i+1} (d-2-i)\binom{d-2}{i-1}.
\]

We prove that these splitting types $b_{j}^{(i)}$ are also scrollar invariants.

\begin{theorem}\label{thm:schreyer.is.scrollar}
Let $C\to \PP^1$ be a cover of degree $d\geq 4$ over a field of characteristic zero or larger than $d$, where $C$ is a nice curve, and take $i\in \{1, \ldots, d-3\}$. Let $e_1'\leq \ldots e_{\beta_i}'$ be the multi-set of scrollar invariants of the partition $(d-i-1, 2, 1^{i-1})$ with respect to $C\to \PP^1$ and let $b_1^{(i)}\leq \ldots \leq b_{\beta_i}^{(i)}$ be the splitting types of the $i$-th syzygy bundle of the Casnati--Ekedahl resolution of $C$. Then for every $j$
\[
e_j = b_j^{(i)}.
\]
\end{theorem}

For a subgroup $H$ of $S_d$, consider the \'etale subalgebra $L^H$ of $L$. This corresponds to another curve equipped with a morphism of degree $[S_d:H]$ to $\PP^1$, which we call the \emph{resolvent curve (with respect to $H$)}, and denote by
\[
\res_H \varphi: \res_H C\to \PP^1.
\]
In \cite{syzrep} we gave a description of the scrollar invariants of such resolvent curves in terms of the scrollar invariants of various partitions $\lambda$. When we remove the simple branching condition from the morphism $C\to \PP^1$, these resolvent curves can become singular or reducible, as will be explained in more detail in section~\ref{sec:resolvents}. However, there is still a notion of scrollar invariants in this context, and the description from \cite{syzrep} carries over as well.

\begin{theorem}\label{thm:scrollar.invs.resolvent}
Let $\varphi: C\to \PP^1$ be a degree $d\geq 4$ morphism over a field $k$ of characteristic zero or larger than $d$, where $C$ is a nice curve. Let $H$ be a subgroup of $S_d$. Then the scrollar invariants of $\res_H \varphi: \res_H C\to \PP^1$ are given by taking the union over all non-trivial partitions $\lambda \vdash d$, of the multi-sets
of scrollar invariants of $\lambda$ with respect to $\varphi$, where each multi-set is to be considered with multiplicity 
 \begin{equation*}   
  \mult(V_\lambda, \Ind^{S_d}_H \mathbf{1}).
\end{equation*}
(Here $\Ind_H^{S_d} \mathbf{1}$ denotes the permutation representation of $S_d/H$.)
\end{theorem}

Since these resolvent curves can be singular or reducible (compared to \cite{syzrep}), it is natural to wonder when precisely this happens. As for being non-singular, this is determined representation theoretically.

\begin{proposition}\label{prop:resolvent.maximal}
Let $\varphi: C\to \PP^1$ be a degree $d\geq 4$ morphism over an algebraically closed field $k$ of characteristic zero or larger than $d$, where $C$ is a nice curve. Let $H$ be a subgroup of $S_d$. Then the resolvent curve $\res_H C$ is non-singular if and only if for every ramification pattern $e$ of $\varphi$ and every partition $\lambda$ of $d$ for which $V_\lambda$ appears in $\Ind_{H}^{S_d}\mathbf{1}$ we have one of the following:
\begin{enumerate}
\item $e = (1^d)$ or $e=(2, 1^{d-2})$, or
\item $\lambda = (d)$ or $\lambda = (d-1,1)$, or
\item $\lambda = (a,b)$ for some $a, b$ and $e = (3, 1^{d-3})$.
\end{enumerate}
In particular, the maximal resolvent $\res_1 C$ is non-singular if and only if $C\to \PP^1$ simply branched.
\end{proposition}

The case (1) in the above proposition corresponds to no ramification or simple branching, while (2) simply expresses the fact that our starting curve $C$ is non-singular. One can think of (3) as classifying the exceptional cases when a resolvent curve is non-singular.

Whether a resolvent curve is reducible or not depends only on the Galois theory of $C\to \PP^1$.

\begin{proposition}\label{prop:resolvent.irre}
Let $\varphi: C\to \PP^1$ be a degree $d\geq 4$ morphism over a field $k$ of characteristic zero or larger than $d$, where $C$ is a nice curve. Let $H$ be a subgroup of $S_d$. Let $G$ be the Galois group of the Galois closure of $k(C)/k(t)$, considered as a subgroup of $S_d$. Then the resolvent curve $\res_H C$ is irreducible (over $k$) if and only if 
\[
GH = S_d.
\]
\end{proposition}

At first sight, it looks like this condition depends on the choice of injection $G\subset S_d$, i.e.\ different choices will lead to conjugate subgroups. And indeed, replacing $G$ by a conjugate subgroup might change the size of the set $GH$. However, if $GH = S_d$ already, then this also holds for every conjugate of $G$.

The applications of~\cite{syzrep} also carry over to this more general setting. In particular, we obtain new bounds on the splitting types. In general, not much is known about the the range in which the splitting types can live, see e.g.\ ~\cite{deopurkar_patel, bopphoff, bopp} for some results around this problem. In fact, even understanding the range in which the scrollar invariants can live is not easy, and much work relates to this question~\cite{bujokas_patel, coppensmartens, sameera, ohbuchi, peikertrosen, ballico, deopurkar_patel_bundles, landesmanlitt}. By studying Maroni bounds on resolvent curves, we obtain some general range in which the splitting types can live. Here is an exemplar statement for degree $4$. We refer to Corollary~\ref{cor:maroni.bound.for.schreyer} for more general bounds in higher degree.

\begin{corollary}\label{cor:bound.splitting.types.degree.4}
Let $C\to \PP^1$ be a degree $4$ morphism over a field of characteristic not $2$ or $3$, where $C$ is a nice curve of genus $g$. Assume that the Galois closure of the corresponding function field extension $K/k(t)$ has Galois group $S_d$ or $A_d$. Then the first splitting types $b_1, b_2$ of the Casnati--Ekedahl resolution of $C\to \PP^1$ satisfy
\[
b_i \leq \frac{2}{3}g+2.
\]
\end{corollary}

\subsection{Main differences with~\cite{syzrep}}\label{sec:differences}

Let us discuss the main differences with \cite{syzrep}, and the new methods required. Let $\varphi: C\to \PP^1$ be a degree $d\geq 4$ cover of curves over a field $k$ of characteristic $0$ or larger than $d$, where $C$ is a nice curve of genus $g$. Write $K/k(t)$ for the corresponding function field extension and let $L/k(t)$ be the Galois closure of $K/k(t)$. For simplicity let us assume that $L/k(t)$ has $S_d$ as its Galois group. In \cite{syzrep} we used the maximal orders in $L/k(t)$, i.e.\ the integral closures of $k[t]$ and $k[t^{-1}]$ inside $L$, to define and study the scrollar invariants corresponding to partitions $\lambda$ of $d$. However, when the original cover $\varphi$ is not simply branched, then almost all of the results in \cite{syzrep} are no longer true with this definition. For example, there is no simple description of the scrollar invariants of hooks $(d-i, 1^i)$, there is no duality statement connecting the scrollar invariants of a partition $\lambda$ with those of its dual partition $\lambda^*$, and there is no simple formula for the genus of $L$. Most importantly however, is that the volume formula, which computes the sum of the scrollar invariants corresponding to a partition, is no longer true. To get around this problem, it is natural to use non-maximal orders in $L/k(t)$, which hence correspond to singular curves. Indeed, this is already clear from the Recilass trigonal construction on a tetragonal curve, which produces singular or reducible curves in the presence of ramification of type $(4)$ or $(2,2)$~\cite{deopurkar_patel, casnati}. The orders we use come from Bhargava and Satriano's theory of \emph{$S_d$-closure}~\cite{bhargavasatriano}, which gives a formal construction of a Galois closure with Galois group $S_d$. With this formalism, practically all of the results from \cite{syzrep} become valid again. As an added bonus, this framework works perfectly well even if the Galois group of $L/k(t)$ is not all of $S_d$. In that case, our orders (and hence the resolvent curves) may become reducible. 

The main difficulty with this new perspective is to properly study this $S_d$-closure in this context, and in particular prove the volume formula. Moreover, a better understanding of the representation theory of the symmetric group is required for this, compared to~\cite{syzrep}. In particular we need results about higher Specht polynomials~\cite{higher_specht_1, higher_specht_2}.

Let us also mention that there is likely an alternative approach possible to generalizing~\cite{syzrep} beyond the simply branched case, based on attaching to each partition $\lambda\vdash d$ the parabolic vector bundle $\mathcal{E}_\lambda$ associated to the representation $V_\lambda$ of $\pi_1^{\mathrm{geom}}(\PP^1\setminus \text{ branch locus of }\varphi)$ via the Mehta--Seshadri correspondence~\cite{MS}. The generalized scrollar invariants of $\lambda$ are precisely the splitting types of $\mathcal{E}_\lambda$. We refer to~\cite{landesmanlitt, landesmanlitt2} for more details. This article sticks closer to the treatment from~\cite{syzrep}.

\subsection{Contents} In section~\ref{sec:prelims} we discuss the preliminaries for this article. We treat some representation theory of $S_d$, in particular on higher Specht polynomials~\cite{higher_specht_1, higher_specht_2}. We also discuss the $S_d$-closure~\cite{bhargavasatriano}. We advise the reader to skip this section on a first reading, and refer back when needed. In section~\ref{sec:scrollars} we define and study the scrollar invariants attached to partitions. One of the main results is the volume formula, Corollary~\ref{cor:volume.formula}. We also prove Proposition~\ref{prop:resolvent.maximal} and~\ref{prop:resolvent.irre} in this section. In section~\ref{sec:resolution} we show how one can construct a certain minimal free resolution Galois-theoretically, and use this to prove Theorem~\ref{thm:schreyer.is.scrollar}. This section is completely analogous to~\cite[Sec.\,4, 6]{syzrep} but replacing the Galois closure by the $S_d$-closure. As such, this section is rather short, and we refer to~\cite[Sec.\,4, 6]{syzrep} for a more extensive treatment. Finally, in section~\ref{sec:examples} we give some applications and examples of this work. Applying the Maroni bound to resolvent curves gives some general bounds on splitting types, in particular proving Corollary~\ref{cor:bound.splitting.types.degree.4}. See Corollary~\ref{cor:maroni.bound.for.schreyer} for a more general statement in every degree involving all splitting types. As for examples, we discuss the most important resolvents in degree $d=4,5,6$. These are the Lagrange resolvent (which geometrically corresponds to the Recillas' trigonal construction), Cayley's sextic resolvent and the exotic resolvent. For each of these, we give a rather complete picture of its geometry, including its non-singular model and the ramification.

\subsection{Acknowledgements} The author thanks Wouter Castryck and Yongqiang Zhao for valuable discussions. The author thanks Mathias Stout for thinking about Propositions~\ref{prop:resolvent.maximal} and~\ref{prop:resolvent.irre}. The author is supported by the Research Foundation -- Flanders (FWO) with grant number 11F1921N.

\section{Preliminaries}\label{sec:prelims}

In this section we recall some representation theoretic results we will need, as well as recalling the $S_d$-closure~\cite{bhargavasatriano} and its properties.

\subsection{Representation theory}\label{sec:repr.theory}

Throughout this section, we fix an integer $d$ and a field $F$ of characteristic zero or larger than $d$. For $2\leq i\leq d-2$ we define the partition $\lambda_i = (d-i, 2, 1^{i-2-1})$ of $d$. We extend this notation by $\lambda_0 = (d), \lambda_1 = (d-1, 1), \lambda_d = (1^d)$ (but we leave $\lambda_{d-1}$ undefined). For any representation $V$, we will denote the corresponding character by $\chi_V$. For the representations $V_\lambda$ of the symmetric group, we use the simpler notation $\chi_\lambda$ instead. If $e$ is another partition of $d$, then by $\chi_\lambda(e)$ we denote the value $\chi_\lambda(\rho)$ where $\rho\in S_d$ is any permutation with cycle type $e$.

We first recall some representation theoretic lemmas from~\cite{syzrep}. No proofs are given, since all proofs can be found in~\cite{syzrep}, or in standard books on representation theory such as~\cite{fultonharris, sagan, hamermesh, serre_repr}.

\begin{lemma}\label{lem: induced and fixed space} 
Let $V$ be an irreducible representation of $S_d$ and let $H \subseteq S_d$ be a subgroup. Then
$\dim V^H = \mult (V, \Ind_H^{S_d} \mathbf{1})$.
\end{lemma}

\begin{lemma}\label{lem: multiplicity fixed}
For any $2\leq i\leq d-2$ we have $\dim \left(V_{\lambda_i}^{S_{\lambda_i}}\right)=1$.
\end{lemma}

\begin{lemma} \label{lem: tensorstandard}
 Consider a partition $\lambda \vdash d$. Then
  \[ V_\lambda \otimes V_{(d-1,1)} \cong \bigoplus_{\mu \vdash d} V_\mu^{c_{\mu \lambda} - \delta_{\mu \lambda}}, \]
  where $c_{\mu \lambda}$ equals the number of ways of transforming $\mu$ into $\lambda$ by removing a cell and adding a cell; here $\delta_{\mu \lambda}$ denotes the Kronecker delta.
\end{lemma}

\begin{lemma} \label{lem: tensornotsostandard}
  Consider a partition $\lambda \vdash d$. Then
  \[ V_{\lambda} \otimes V_{(d-2,2)} \cong \bigoplus_{\mu \vdash d} V_\mu^{-c_{\mu \lambda} + \frac{1}{2} (d_{\mu \lambda} + e_{\mu \lambda} - e'_{\mu \lambda} ) }, \]
  where $c_{\mu \lambda}$, $d_{\mu \lambda}$, $e_{\mu \lambda}$, $e'_{\mu \lambda}$ denote
  the number of ways of transforming $\mu$ into $\lambda$ by 
  \begin{itemize}
    \item removing a cell and adding a cell,
    \item consecutively removing two cells and consecutively adding two cells,
    \item removing two horizontally adjacent cells and adding two horizontally adjacent cells, or removing two vertically adjacent cells and adding two vertically adjacent cells,
    \item removing two horizontally adjacent cells and adding two vertically adjacent cells, or removing two vertically adjacent cells and adding two horizontally adjacent cells,
  \end{itemize}
  respectively.
\end{lemma}

\begin{corollary} \label{cor: multiplicity double next}
We have
\[ \mult(V_{\lambda_{i+2}}, V_{\lambda_i}\otimes \Sym^2 V_{(d-1, 1)} ) = \left\{ \begin{array}{ll}
0 & \text{for all $2 \leq i \leq d-4$,} \\
1 & \text{if $i = d-2$.} \\ \end{array} \right. \]
\end{corollary}

\begin{lemma} \label{lem: sym3}
Assume $d \geq 5$. Then 
\[ \mult(V_{\lambda_3}, \Sym^3 V_{(d-1,1)})  = 
\mult(V_{\lambda_d}, V_{\lambda_{d-3}} \otimes \Sym^3 V_{(d-1,1)}) = 0.\]
\end{lemma}

For $d=3, 4$ one also needs the following ad-hoc results. These follow from a simple computation of their characters.
\begin{equation} \label{eq: adhocmults} 
  \mult (V_{(1^3)}, \Sym^3 V_{(2,1)}) = 1 \quad \text{and} \quad
 \mult (V_{(1^4)}, \Sym^4 V_{(3, 1)}) = 0,
\end{equation}

Compared to~\cite{syzrep}, we need some more advanced results about the representation theory of the symmetric group. 

Firstly, we will need some results regarding reading words of Young tableaux. For a partition $\lambda \vdash d$, a \emph{Young tableau of shape $\lambda$} is the Young diagram of $\lambda$ with the integers $1, 2, \ldots, d$ filled in. We call a Young tableau $T$ \emph{standard} if the entries in each row and column of $T$ are increasing. Denote by $\ST(\lambda)$ the set of standard Young tableaux of shape $\lambda$, and recall that there are $\dim V_\lambda$ of these. Let $T$ be a standard Young tableau. Then the \emph{reading word} $w(T)$ of $T$ is obtained by concatenating all entries in $T$ from left to right and bottom to top. For example, the reading word of the standard tableau
\begin{center}
$T = \,$ \begin{ytableau}
1 & 3 & 5 & 7 \\ 
2 & 4 \\
6
\end{ytableau}
\end{center}
is $6\,2\,4\,1\,3\,5\,7$. The \emph{charge word of $T$}, denoted by $\chw(T)$, is obtained by inductively assigning subscripts $s(i)$ to the letters $i=1, \ldots, d$ in the reading word of $T$ in the following way. Assign the letter $1$ the subscript $0$. For $i\geq 1$ assign the letter $i+1$ the subscript $s(i+1)=s(i)+1$ if $i+1$ is to the left of $i$ and the subscript $s(i+1)=s(i)$ if $i+1$ is to the right of $i$. The \emph{charge of $T$}, denoted by $\ch (T)$, is the sum of the subscripts $s(i)$ of the charge word of $T$. For the Young tableau as above, the charge word of $T$ is equal to $6_3\, 2_1\, 4_2\, 1_0\, 3_1\, 5_2\, 7_3$, and it has charge $3+1+2+0+1+2+3 = 12$. We sometimes also write down $\chw(T)$ as a Young tableau, where the subscript $s(i)$ is written in the cell containing $i$. For example, for $T$ as above, we have
\begin{center}
$\chw(T) = \,$\begin{ytableau}
0 & 1 & 2 & 3 \\
1 & 2 \\
3
\end{ytableau}\,.
\end{center}
We denote by $s(T,i)$ the subscript of the letter $i$ in the charge word of $T$ and by $m(T)$ the largest subscript appearing in the charge word of $T$. 

\begin{lemma}\label{lem:plambda.reading.words}
Let $\lambda\vdash d$. For each $i=1, \ldots, d-1$ the number of $T\in \ST(\lambda)$ for which $s(T,i+1) = s(T,i) + 1$ is equal to 
\[
p(\lambda) = \frac{1}{2} \left( \dim V_\lambda - \chi_\lambda((1\,2))\right).
\]
\end{lemma}

\begin{proof}
The cases $d=1, 2$ are easy to do explicitly, so assume that $d\geq 3$.

We first note that $p(\lambda) = \sum_{\mu < \lambda} p(\mu)$ where the sum is over all partitions $\mu\vdash d-1$ which are obtained from $\lambda$ by removing one cell from the Young diagram. Hence, by using induction, Young's rule and restricting to $S_{d-1}$, it suffices to prove the case $i=d-1$. 

Firstly, note that $s(T, d) = s(T,d-1) + 1$ if and only if the letter $d$ appears strictly below $d-1$ in the tableau $T$. So we wish to prove that the number of $T\in \ST(\lambda)$ for which the letter $d$ appears strictly below $d-1$ in $T$ is equal to $p(\lambda)$. Let $\mu\vdash d-2$ be a partition obtained from $\lambda$ by removing two cells in different rows of the Young diagram. For any Young tableau $T'\in \ST(\mu)$ there is exactly one way to obtain a $T\in \ST(\lambda)$ for which $T'\subset T$ and such that $d$ appears below $d-1$ in $T$. Moreover, every standard Young tableaux $T$ for which $d$ appears strictly below $d-1$ can be obtained in this way. Let $M_1, M_2$ and $M_3$ be the set of partitions $\mu\vdash d-2$ obtained from $\lambda$ by removing two horizontally adjacent cells, two vertically adjacent cells and two non-adjacent cells respectively. Then the quantity we are looking for is given by
\[
\sum_{\mu\in M_2\cup M_3} \dim V_\mu
\]
i.e.\ the sum is over all partitions $\mu\vdash d-2$ which are obtained from $\lambda$ by removing two cells in different rows. Now note that by Young's rule
\[
\dim V_\lambda = \dim \Res_{S_{d-2}}^{S_d} V_\lambda = \sum_{\mu\in M_1} \dim V_\mu + \sum_{\mu\in M_2} \dim V_\mu +2 \sum_{\mu\in M_3} \dim V_\mu.
\]
(The coefficient for $M_3$ is $2$ since there are two ways to remove two non-adjacent cells from $\lambda$.) The Murnaghan--Nakayama rule~\cite[Thm.\,4.10.2]{sagan} implies that
\[
\chi_\lambda((1\,2)) = \sum_{\mu\in M_1}\dim V_\mu - \sum_{\mu\in M_2}\dim V_\mu.
\]
We conclude that
\[
\sum_{\mu\in M_2\cup M_3} \dim V_\mu = \frac{1}{2} \left( \dim V_\lambda - \chi_\lambda((1\,2))\right),
\]
as desired.
\end{proof}

The following result will be needed for proving our volume formula for scrollar invariants attached to representations.

\begin{corollary}\label{cor:sum.maximal.reading.word}
Let $\lambda\vdash d$. We have that
\[
\sum_{T\in \ST(\lambda)} m(T) = p(\lambda)(d-1).
\]
\end{corollary}

\begin{proof}
We have that
\begin{align*}
\sum_{T\in \ST(\lambda)} m(T) &= \sum_{T\in \ST(\lambda)} s(T, d) = \sum_{T\in \ST(\lambda)} \sum_{i=1}^{d-1} (s(T, i+1) - s(T,i)) \\
&= \sum_{i=1}^{d-1}\sum_{T\in \ST(\lambda)} (s(T,i+1) - s(T,i)) = p(\lambda)(d-1),
\end{align*}
where we have used Lemma~\ref{lem:plambda.reading.words}.
\end{proof}

To determine precisely when a resolvent is maximal, we will need the following technical lemma.

\begin{lemma}\label{lem:repr.lemma.for.maximality}
Let $\lambda, e$ be partitions of $d$. Then 
\begin{equation}\label{eq:repr.inequality.maximality}
\frac{\dim V_\lambda - \chi_\lambda(e)}{\dim V_\lambda - \chi_\lambda ((12))} \leq \frac{\dim V_{(d-1,1)} - \chi_{(d-1,1)}(e)}{\dim V_{(d-1,1)} - \chi_{(d-1,1)} ((12))},
\end{equation}
with equality if and only if
\begin{itemize}
\item $\lambda = (d)$ or $\lambda = (d-1,1)$, or
\item $e = (1^d)$ or $e = (2, 1^{d-2})$, or
\item $\lambda = (a,b)$ for some $a,b$ and $e = (3, 1^{d-3})$.
\end{itemize}
\end{lemma}

When $\lambda = (d)$, the denominator on the left hand side becomes $0$, so in that case the inequality should be understood by multiplying out the fractions.

For the proof, we need some more representation theoretic notions. Let $\mu, \lambda$ be Young diagrams such that $\mu\subset \lambda$. The \emph{skew Young diagram $\lambda/\mu$} is the diagram obtained by removing all cells in $\mu$ from $\lambda$. A \emph{border strip} is a connected skew Young diagram that does not contain a $2\times 2$ square of cells. Let $\alpha$ be a skew Young diagram with $m$ cells. The \emph{height} $\hto(\alpha)$ of a skew Young diagram $\alpha$ is one less than its number of rows. For example, the skew Young diagram $(5, 3, 2, 1)/ (2, 1)$ is a border strip of height 3.
\begin{center}
$(5,3,2,1)/(2,1)=\,$ \ydiagram{2+3, 1+2, 2, 1}
\end{center}

\begin{lemma}\label{lem:repr.bounding.char.by.plambda}
Let $\lambda\vdash d$ and $2\leq m \leq d-1$. Then we have that
\[
\chi_\lambda( (m, 1^{d-m})) - \chi_\lambda((m+1, 1^{d-m-1})) \leq p(\lambda).
\]
\end{lemma}

\begin{proof}
By induction on $d$ and restricting to $S_{d-1}$ we only have to check the case that $m+1 = d$.  By the Murnaghan--Nakayama rule~\cite[Thm.\,4.10.2]{sagan} we have that $\chi_\lambda((d))$ is $0$ if $\lambda$ is not a hook, while it is $(-1)^{\hto(\lambda)}$ if it is a hook. Let $\lambda' = \lambda/(1)$, i.e.\ we remove the upper-left cell from $\lambda$. Then similarly, $\chi_\lambda((d-1,1))$ is $0$ if $\lambda'$ is not a border strip, while it is $(-1)^{\hto(\lambda')}$ if is a border strip. In particular, we have that
\[
\chi_\lambda( (d-1,1) ) - \chi_\lambda( (d) ) \leq 2.
\]
On the other hand, by induction we have for $d\geq 5$ and $\lambda\neq (d), (d-1,1)$ that $p(\lambda)\geq 2$. The cases that $d\leq 4$ or $\lambda = (d), (d-1,1)$ are easy to check by hand, and the result follows. 
\end{proof}

\begin{proof}[Proof of Lemma~\ref{lem:repr.lemma.for.maximality}]
We will prove this lemma by induction on $d$ and the number of fixed points of $e$. The case $d\leq 2$ is easily checked by hand, so assume that $d\geq 3$.

First, note that the right-hand side in equation~\eqref{eq:repr.inequality.maximality} is equal to
\[
\frac{d - \#\mathrm{Fix}(e)}{2},
\]
where $\# \mathrm{Fix}(e)$ is the number of fixed points of the permutation $e$. 

Assume first that $e$ has a fixed point. Then we may assume that $e\in S_{d-1}$. By induction on $d$ we have for every $V_\gamma\in \Res_{S_{d-1}}^{S_d} V_\lambda$ that
\[
\frac{\dim V_\gamma - \chi_\gamma(e)}{\dim V_\gamma - \chi_\gamma ((12))} \leq \frac{d-\#\mathrm{Fix}(e)}{2}.
\]
Hence by Young's rule, we will also have that
\[
\frac{\dim V_\lambda - \chi_\lambda(e)}{\dim V_\lambda - \chi_\lambda ((12))} \leq \frac{d-\#\mathrm{Fix}(e)}{2}.
\]
Moreover, if for some $V_\gamma\in \Res_{S_{d-1}}^{S_d} V_\lambda$ there is a strict inequality here, then the same holds for $\lambda$.

Now assume that $e$ has no fixed points. Write $e = (e_1, \ldots, e_s)$, with $e_1\geq \ldots \geq e_s$, $e' = (e_1, \ldots, e_{s-1})$ and $e_s = m\geq 2$. Rewriting equation~\eqref{eq:repr.inequality.maximality}, we wish to prove that
\[
\chi_\lambda(1^d) - \chi_\lambda(e) \leq p(\lambda) (d-\#\mathrm{Fix}(e)) = p(\lambda) d.
\]
By induction on $\# \mathrm{Fix}(e)$, we have for the left hand side
\begin{equation}\label{eq:induction.inequality.repr.maximality}
\chi_\lambda(1^d) - \chi_\lambda((e', 1^m)) + \chi_\lambda((e', 1^m)) - \chi_\lambda(e) \leq p(\lambda)(d-m) + \chi_\lambda((e', 1^m)) - \chi_\lambda(e).
\end{equation}
Hence it is enough to prove that
\[
\chi_\lambda((e', 1^m)) - \chi_\lambda(e) \leq mp(\lambda).
\]
By the Murnaghan--Nakayama rule~\cite[Thm.\,4.10.2]{sagan}, the left hand side is equal to 
\[
\sum_{\lambda'\vdash d-m} c_{\lambda, \lambda'} \chi_{\lambda'}(e),
\]
where the sum is over all partitions of $d-m$ and
\[
c_{\lambda, \lambda'} = \mult (\lambda', \Res_{S_{d-m}}^{S_d}V_\lambda) + \begin{cases}
(-1)^{\hto(\lambda/\lambda') + 1} & \text{ if $\lambda/\lambda'$ is a border strip,} \\
0 &  \text{ if $\lambda/\lambda'$ is not a border strip.}
\end{cases}
\]
Since $|\chi_{\lambda'}(e')| \leq \dim V_{\lambda'}$ and $c_{\lambda, \lambda'}\geq 0$, it is enough to prove that 
\begin{equation}\label{eq:maximality.dim.inequality}
\sum_{\lambda' \vdash d-m} c_{\lambda, \lambda'} \dim V_{\lambda'} \leq mp(\lambda).
\end{equation}
As in the proof of Lemma~\ref{lem:plambda.reading.words} denote by $M_1, M_2$ and $M_3$ the partitions $\nu$ of $d-2$ which are obtained from $\lambda$ by removing two horizontally adjacent cells, two vertically adjacent cells and two non-adjacent cells respectively. Then using the characterization of $p(\lambda)$ from the proof of Lemma~\ref{lem:plambda.reading.words} and rewriting equation~\eqref{eq:maximality.dim.inequality}, we wish to prove that
\ytableausetup{smalltableaux}
%\begin{align}\label{eq:repr.bst.inequality}
%(m-1) \sum_{\lambda/\nu = \vtabl} &\dim V_\nu + (m-2) \sum_{\lambda/\nu = 2\btabl} \dim V_\nu  - \sum_{\lambda/\nu = \htabl} \dim V_\nu \\
%&\geq \sum_{\substack{\lambda'\vdash d-m \\ \lambda/\lambda' \text{ border strip}}} (-1)^{\hto(\lambda/\lambda') + 1} \dim V_{\lambda'}, \nonumber
%\end{align}
\begin{align}\label{eq:repr.bst.inequality}
(m-1) \sum_{\nu \in M_1} &\dim V_\nu + (m-2) \sum_{\nu \in M_3} \dim V_\nu  - \sum_{\nu \in M_2} \dim V_\nu \\
&\geq \sum_{\substack{\lambda'\vdash d-m \\ \lambda/\lambda' \text{ border strip}}} (-1)^{\hto(\lambda/\lambda') + 1} \dim V_{\lambda'}. \nonumber
\end{align}

We prove this inequality by induction on $m$. For $m=2$, it is straightforward to check, so assume that $m\geq 3$. The Murnaghan--Nakayama rule~\cite[Thm.\,4.10.2]{sagan} shows that the right hand side of equation~\eqref{eq:repr.bst.inequality} is equal to $-\chi_\lambda( (m, 1^{d-m}) )$. By induction on $m$, it is therefore enough to prove that
%\[
%\chi_\lambda( (m-1, 1^{d-m}) ) - \chi_\lambda( (m, 1^{d-m}) ) \leq \sum_{\lambda/\nu = \vtabl} \dim V_\nu + \sum_{\lambda/\nu = 2\btabl} \dim V_\nu = p(\lambda).
%\]
\[
\chi_\lambda( (m-1, 1^{d-m}) ) - \chi_\lambda( (m, 1^{d-m}) ) \leq \sum_{\nu \in M_1} \dim V_\nu + \sum_{\nu \in M_3} \dim V_\nu = p(\lambda).
\]
But this is precisely Lemma~\ref{lem:repr.bounding.char.by.plambda}, proving the desired inequality.

Now assume that we have equality, i.e.\
\[
\chi_\lambda((1^d)) - \chi_\lambda(e) = p(\lambda)(d-\# \mathrm{Fix}(e)).
\]
If $e$ has a fixed point, then we can reason similarly as above to conclude by induction on $d$. So suppose that $e$ has no fixed points. Then in equation~\ref{eq:induction.inequality.repr.maximality} we must have that 
\[
\chi_\lambda((1^d)) - \chi_\lambda((e', 1^m)) = p(\lambda)(d-m).
\]
By induction on $m$, this means that either we have that $\lambda = (d), (d-1,1)$, or that $e' = (1^{d-m}), (2,1^{d-m-2})$, or that $\lambda = (a,b)$ and $e' = (3, 1^{d-m-3})$. If $\lambda = (d), (d-1,1)$ then we don't have to check anything, so we go through the other cases. 

Assume that $e' = (1^{d-m})$. If $m<d$, then $2\leq m=e_s \leq e_{s-1}=1$ gives a contradiction. So $m=d$ and $e = (d)$. But in that case, it follows by induction on $d$, that for $d\geq 4$ we have that
\[
\chi_\lambda((1^d)) - \chi_\lambda((d)) \leq 3 p (\lambda) < dp(\lambda),
\]
as desired. The cases $d\leq 3$ are easy to check explicitly.

Now assume that $e' = (2, 1^{d-m-2})$. If $d-m-2>0$ then again $m = e_s \leq 1$, contrary to assumption. So $d-m-2 = 0$ and $m = 2$. This implies that $d=4$, and all cases there are easy to do explicitly.

Finally, assume that $e' = (3, 1^{d-m-3})$. Again, if $d-m-3>0$ then we obtain a contradiction. Hence $d\leq 6$ and $m\leq 3$. All cases here are easy to check, proving the result.

\end{proof}

\ytableausetup{nosmalltableaux}

\subsection{Higher Specht polynomials}\label{sec:higher.specht}

Consider the ring $F[x_1, \ldots, x_d]$ with its natural $S_d$-action permuting the variables. Let $\lambda$ be a partition of $d$. Then the \emph{Specht polynomial}~\cite{specht} corresponding to $\lambda$ is a certain explicit polynomial in $F[x_1, \ldots, x_d]$ generating a representation isomorphic to $V_\lambda$. Moreover, the degree of this polynomial will be the lowest degree in which a copy of $V_\lambda$ appears in the $S_d$-representation $F[x_1, \ldots, x_d]$~\cite[p.\,95]{campbellinvariant}.

We will need results concerning \emph{higher Specht polynomials}, which are a generalization of these Specht polynomials. Consider the \emph{ring of coinvariants}
\[
\Rco = \frac{F[x_1, \ldots, x_d]}{(s_1, \ldots, s_d)},
\]
where the $s_i$ are the elementary symmetric polynomials in the $x_i$, i.e.\ $s_1 = x_1 + x_2 + \ldots x_d$, $s_2 = x_1x_2 + x_1x_3 + \ldots$ and so on. The symmetric group $S_d$ still acts on this ring by permuting the variables $x_i$, turning it into an $S_d$-representation. It is known that with this $S_d$-action, $\Rco$ is isomorphic to the regular representation $F[S_d]$~\cite{stanley_coinvariants}. Hence as an $S_d$-representation it decomposes as $\Rco = \bigoplus_\lambda W_\lambda$ where $W_\lambda \cong V_\lambda^{\dim V_\lambda}$ is the isotypic component corresponding to $V_\lambda$. The higher Specht polynomials will give us an explicit decomposition for $\Rco$ into its irreducible components $V_\lambda$. In more detail, to each pair of standard Young tableaux $(S,T)$ of the same shape $\lambda$, we will associate an element $F_T^S \in \Rco$ generating a copy of $V_\lambda$, and all these elements together will form an $F$-basis for $\Rco$.

So let $\lambda$ be a partition of $d$ and let $S, T\in\ST(\lambda)$. Define the monomial
\[
x_T^{\chw S} = \prod_{i=1}^d x_i^{\chw(i, S, T)},
\]
where $\chw(i,S,T)$ is the subscript of the charge word of $S$, at the cell in $S$ corresponding to the cell containing $i$ in $T$. For example, for $S, T$ as given
\begin{center}
$S = $
\begin{ytableau}
1 & 3 & 5 & 7 \\
2 & 4 \\
6
\end{ytableau}
$\quad\chw(S) = $
\begin{ytableau}
0 & 1 & 2 & 3 \\
1 & 2 \\
3
\end{ytableau}
$\quad T = $
\begin{ytableau}
1 & 2 & 3 & 6 \\
4 & 5 \\
7
\end{ytableau}
\end{center}
we have that $\chw(S)= 6_3 2_1 4_2 1_0 3_1 5_2 7_3$ and that 
\[
x_T^{\chw S} = x_2 x_3^2 x_4 x_5^2 x_6^3 x_7^3.
\]
Let $C(T)\subset S_d$ be the group of \emph{column permutations of $T$}, consisting of those permutations in $S_d$ which map every integer $i\in \{1, \ldots, d\}$ to another in the same column of $T$. Similarly define the \emph{row permutations of $T$}, denoted by $R(T)\subset S_d$. We define 
\[
\varepsilon_T = \sum_{\tau \in C(T)} \sum_{\sigma \in R(T)} \sgn(\tau) \tau \sigma \in F[S_d].
\]
Then the \emph{higher Specht polynomial $F_T^S$} is defined as $F_T^S = \varepsilon_T x_T^{\chw S}\in \Rco$. 

Note that $F_T^S$ is a $\ZZ$-linear combination of conjugates of the element
\[
x_1^{s(S, 1)}x_2^{s(S,2)} \cdots x_d^{s(S, d)}.
\]
The main theorem about higher Specht polynomials is the following canonical decomposition of $\Rco$ into its irreducible components.

\begin{theorem}{\cite[Thm.\,1]{higher_specht_2}}\label{thm:higher.specht.poly}
Let $d$ be a positive integer.
\begin{enumerate}
\item For every partition $\lambda$ of $d$ and every $S\in \ST(\lambda)$, the space
\[
\operatorname{span}\{ F_T^S \mid T\in\ST(\lambda)\} \subset \Rco
\]
is an irreducible $S_d$-representation isomorphic to $V_\lambda$.
\item The set $\{F_T^S\mid S, T \in\ST(\lambda)\}$ is an $F$-basis of $\Rco$.
\end{enumerate}
\end{theorem}

\subsection{The $S_d$-closure}

In this section we review the construction and basic properties of the $S_d$-closure construction of Bhargava--Satriano \cite{bhargavasatriano}. Throughout, all rings are commutative with $1$. Let $A$ be a ring of rank $d$ over a base ring $B$, by which we mean a $B$-algebra which is free of rank $d$ as a $B$-module. For us, $B$ will typically be $k[t]$ and $A$ will be the integral closure of $k[t]$ inside a function field over $k(t)$. The $S_d$-closure will be a certain ring $G(A/B)$ with an action of $S_d$ which is a natural model for a Galois closure of $A$ over $B$ with Galois group $S_d$. 

For an element $a\in A$ let 
\[
P_a(x) = x^d - s_1(a)x^{d-1} + s_2(a)x^{d-2} + \ldots + (-1)^d s_d(a)
\]
be the characteristic polynomial of the multiplication by $a$ map $A\to A$. We denote by $a^{(1)}, a^{(2)}, \ldots, a^{(d)}$ the elements $a\otimes 1\otimes \ldots \otimes 1, 1\otimes a\otimes 1 \otimes \ldots \otimes 1, \ldots,  1\otimes \ldots \otimes 1\otimes a$ in $A^{\otimes d}$. These will play the role of the conjugates of $a$. Denote by $I(A,B)$ the ideal in $A^{\otimes d}$ generated by all expressions
\[
s_j(a) - \sum_{1\leq i_1< i_2 < \ldots < i_j \leq d} a^{(i_1)}\cdots a^{(i_j)},
\]
for $a\in A$ and $j\in \{1, \ldots, d\}$. 

We define the \emph{$S_d$-closure of $A/B$}, denoted by $G(A/B)$ as the $B$-algebra 
\[
G(A/B) = A^{\otimes d} / I(A,B).
\]
There is a natural action of $S_d$ on $G(A/B)$ by permuting the coordinates in $A^{\otimes d}$.

There is also a naturally defined \emph{trace map}
\[
\Tr_{G(A/B)/B}: G(A/B) \to B: \alpha \mapsto \sum_{\sigma\in S_d} \sigma(\alpha).
\]

If $A/B$ is a degree $d$ extension of fields with Galois group $S_d$, then the $S_d$-closure $G(A/B)$ is simply the Galois closure of $A/B$. More generally, if $\tilde{A}$ is the Galois closure of $A/B$ and $\Gal(\tilde{A}/B) = H\subset S_d$, then 
\[
G(A/B) \cong \Ind_{H}^{S_d} \tilde{A} 
\]
as $B[S_d]$-algebras \cite[Thm.\,2]{bhargavasatriano}. As a $B$-algebra, this is simply $\tilde{A}^r$ where $r$ is the index of $H$ in $S_d$.

The most important property of the $S_d$-closure is that it commutes with base change. 

\begin{theorem}{\cite[Thm.\,1]{bhargavasatriano}}\label{thm:Sdclosure.basechange}
If $A/B$ is a ring of rank $d$ and $C$ is a $B$-algebra then there is a natural isomorphism
\[
G(A/B)\otimes_B C\cong G((A\otimes_B C)/C)
\]
as $C$-algebras.
\end{theorem}

For monogenic extensions and products of rings we have the following.

\begin{theorem}{\cite[Thm.\,16]{bhargavasatriano}}\label{thm:Sdclosure.monogenic}
Let $f$ be a monic polynomial of degree $d$ in $B[x]$ and let $A = B[x]/(f(x))$. Then the $S_d$-closure $G(A/B)$ is of rank $d!$ over $B$, with a basis consisting of all monomials
\[
\prod_{i=1}^d x_i^{e_i}
\]
where the exponents satisfy $0\leq e_i < i$, and $x_i$ denote the conjugates of $x=x_1$ in $G(A/B)$.
\end{theorem}

\begin{theorem}{\cite[Thm.\,6]{bhargavasatriano}}\label{thm:Sdclosure.product}
Let $A_1, \ldots, A_\ell$ be rings of rank $d_1, \ldots, d_\ell$ over $B$. Then
\[
G(A_1\times \ldots \times A_\ell/B) \cong \Ind_{S_{d_1}\times \ldots \times S_{d_\ell}}^{S_d} \left(G(A_1/B)\otimes_B \ldots \otimes_B G(A_\ell/B)\right)
\]
as $B[S_d]$-algebras.
\end{theorem}

So as $B$-algebras, this gives that 
\[
G(A_1\times \ldots \times A_\ell/B) \cong  \left(G(A_1/B)\otimes_B \ldots \otimes_B G(A_\ell/B)\right)^{\binom{d}{d_1, \ldots, d_\ell}}.
\]

\section{Scrollar invariants of representations}\label{sec:scrollars}

In this section we introduce scrollar invariants of irreducible representations (or equivalently partitions) and prove various results about them. We first set some notation.

Throughout, we fix an integer $d\geq 4$ and let $k$ be a field of characteristic zero or larger than $d$. We let $\varphi: C\to \PP^1$ be a degree $d$ cover of curves over $k$, where $C$ is a nice curve of genus $g$. We denote by $K/k(t)$ the function field extension corresponding to $\varphi$ and let $L = G(K/k(t))$ be the $S_d$-closure of this extension. This is a $k(t)$-algebra of rank $d!$, with a natural action of $S_d$. Let $\OO_K$ be the integral closure of $k[t]$ in $K$, and similarly let $\OO_{K, \infty}$ be the integral closure of $k[t^{-1}]$ in $K$. We will call $\OO_K$ (and $\OO_{K, \infty}$) the \emph{ring of integers of $K$}. We define the \emph{ring of integers of $L$} to be $\OO_L = G(\OO_K/k[t])$ and similarly $\OO_{L, \infty} = G(\OO_{K, \infty}/k[t^{-1}])$. It is important to note that $\OO_L$ is typically \emph{not} integrally closed and hence not equal to maximal order of $L$. In fact, we will later prove that this happens if and only if the morphism $\varphi: C\to \PP^1$ is simply branched. 

\subsection{The rank of the ring of integers}

Our first goal is to prove that $\OO_L$ is a ring of rank $d!$ over $k[t]$, so that we can consider $\OO_L$ as a subring of $L$. In general, it can happen that the rank of the $S_d$-closure of a rank $d$ ring extension $A/B$ is strictly larger than $d$, see \cite[Thm.\,9]{bhargavasatriano}.

\begin{proposition}
With notation as above, $\OO_L$ (resp.\, $\OO_{L, \infty}$) is a rank $d!$ ring over $k[t]$ (resp.\, $k[t^{-1}]$), which is naturally contained in $L$.
\end{proposition}

In other words, $\OO_L$ is an order in $L$, which is not necessarily maximal.

\begin{proof}
Since $k[t]$ is a PID, we can decompose $G(\OO_K/k[t])$ as a $k[t]$-module as follows
\[
G(\OO_K/k[t]) \cong k[t]^r \oplus \frac{k[t]}{(q_1)}\oplus \ldots \oplus \frac{k[t]}{(q_n)}
\]
where $q_i\neq 0$ for any $i$. By Theorem~\ref{thm:Sdclosure.basechange}, if we base change $G(\OO_K/k[t])$ to $k(t)$ we obtain $L$, which is a ring of rank $d!$ over $k(t)$. Hence $r=d!$. We wish to show that there is no torsion, i.e.\ $n=0$. For this, note that $G(\OO_K/k[t])$ has torsion if and only if $G(\OO_K/k[t])\otimes_{k[t]} k^{\alg} [t]$ has torsion, where $k^{\alg}$ is an algebraic closure of $k$. So we may assume that $k$ is algebraically closed. Suppose that $n>0$ and let $p$ be a prime in $k[t]$ dividing $q_1$. Then $G(\OO_K/k[t])\otimes_{k[t]} k[t]/(p)$ is a ring of rank strictly larger than $d!$. Now, there are integers $e_1, \ldots, e_m$ such that, as $k[t]/(p)$-algebras, we have that
\[
\frac{\OO_K}{(p)} \cong \frac{k[t]}{(t^{e_1})} \oplus \ldots \frac{k[t]}{(t^{e_m})}.
\]
But then $G(\frac{\OO_K}{(p)} / \frac{k[t]}{(p)} )$ is the $S_d$-closure of a product of monogenic extensions, and so by Theorems \ref{thm:Sdclosure.monogenic} and \ref{thm:Sdclosure.product} this is a ring of rank $d!$. This is a contradiction, showing that $G(\OO_K/k[t])$ is of rank $d!$ and the natural morphism $G(\OO_K/k[t])\to L$ is injective.
\end{proof}

\subsection{Scrollar invariants}

Let $V$ be a $k(t)$-vector space inside $L$ of dimension $n$ and put $\OO_V = V\cap \OO_L$ and $\OO_{V, \infty} = V\cap \OO_{L, \infty}$. Then $\OO_V$ is a free $k[t]$-module of rank $n$, and similarly $\OO_{V, \infty}$ is a free $k[t^{-1}]$-module of rank $n$. Hence by \cite[Cor.\,4.3]{hessRR} there exist unique integers $0\leq r_1 \leq \ldots \leq r_n$ together with a $k[t]$-basis $v_1, \ldots, v_n$ of $\OO_V$ such that
\[
t^{-r_1}v_1, \ldots, t^{-r_n}v_n
\]
is a $k[t^{-1}]$-basis of $\OO_{V, \infty}$. We call the $r_i$ the \emph{scrollar invariants} of $V$ and a basis $v_1, \ldots, v_n$ a \emph{reduced basis}. The scrollar invariants of $K$ are $0, e_1, \ldots, e_{d-1}$, where the $e_i$ are the usual scrollar invariants associated to the morphism $\varphi: C\to \PP^1$ as in equation~\eqref{eq:splitting.varphi*OC}.

Recall that $L$ is isomorphic to the regular representation $k(t)[S_d]$. We follow the notation from before and denote for $\lambda\vdash d$ the corresponding isotypic component in $L$ by $W_\lambda$. We set $\OO_\lambda = W_\lambda \cap \OO_L$ and similarly $\OO_{\lambda, \infty} = W_\lambda\cap \OO_{L, \infty}$. We will be interested in the scrollar invariants of $W_\lambda$. Crucial to our discussion is the following.

\begin{theorem}
The scrollar invariants of $W_\lambda$ form a multi-set of the form
\[
e_{\lambda, 1}, \ldots, e_{\lambda, 1}, \qquad e_{\lambda, 2}, \ldots, e_{\lambda, 2}, \qquad e_{\lambda, \dim V_\lambda}, \ldots, e_{\lambda, \dim V_\lambda},
\]
where each block contains $\dim V_\lambda$ elements.
\end{theorem}
\begin{proof}
This follows along the same lines as \cite[Sec.\,5]{syzrep}.
\end{proof}

We call the $e_{\lambda, 1}, e_{\lambda, 2}, \ldots, e_{\lambda, \dim V_\lambda}$ the \emph{scrollar invariants of $\lambda$ with respect to $L$}, or with respect to $\varphi$.

We recall a little lemma which is useful in proving that a candidate reduced basis is in fact reduced.

\begin{lemma}\label{lem:candidate.basis.is.reduced}
Let $V$ be a $k(t)$-vector subspace of $L$ with scrollar invariants $r_1\leq \ldots \leq r_n$. Suppose that the elements $v_1, \ldots, v_n\in \OO_V$ form a $k(t)$ basis for $V$, that there are integers $r_1'\leq \ldots \leq r_n'$ such that
\[
r_1 + \ldots + r_n = r_1' + \ldots + r_n'
\]
and that $t^{-r_i'}v_i$ is in $\OO_{V, \infty}$. Then $r_i=r_i'$ for all $i$ and $v_1, \ldots, v_n$ is a reduced basis for $V$.
\end{lemma}

\begin{proof}
See \cite[Lem.\,24]{syzrep}.
\end{proof}

\subsection{Resolvent curves}\label{sec:resolvents} We next show that a reduced basis for $L$ itself can be obtained by taking reduced bases for each $W_\lambda$. In fact, we prove a more general statement about the scrollar invariants of \emph{resolvent curves}. For $H$ a subgroup of $S_d$ let $\OO_L^H$ be the subring of $\OO_L$ consisting of the elements fixed by $H$, and define $\OO_{L, \infty}^H$ similarly. Let $v_1, \ldots, v_n$ be a reduced basis of $\OO_L^H$ with scrollar invariants $r_1, \ldots, r_n$. Then we define the \emph{resolvent curve $\res_H C$} by gluing together $\Spec \OO_L^H$ and $\Spec \OO_{L, \infty}^H$ via the bases $v_1, \ldots, v_n$ and $t^{-r_1}v_1, \ldots, t^{-r_n}v_n$. This is a curve, i.e.\ it is projective and reduced, and comes naturally equipped with a degree $[S_d:H]$ morphism 
\[
\res_H \varphi: \res_H C\to \PP^1
\]
coming from the inclusion $k[t]\subset \OO_L^H$. It may happen that $\res_H C$ is singular or reducible, and we will later classify precisely when this happens.

When we talk about the scrollar invariants of $\res_H C$, we will mean the scrollar invariants $r_1, \ldots, r_n$ as above, i.e.\ those of $\OO_L^H$. We denote by $\Ind_H^{S_d} \mathbf{1}$ the representation of $S_d$ induced by the trivial representation of $H$, in other words this is the coset representation $k(t)[S_d/H]$. One can now prove Theorem~\ref{thm:scrollar.invs.resolvent}, which expresses the scrollar invariants of $\res_H C$ in terms of the scrollar invariants of various partitions $\lambda$.

\begin{proof}[Proof of Theorem~\ref{thm:scrollar.invs.resolvent}]
The proof is identical to \cite[Thm.\,3]{syzrep}
\end{proof}

\subsection{The $S_e$-closure of $k[t^{1/e}]$}\label{sec:Se.closure.k[t1/e]}

Many of the proofs below will proceed by working geometrically and locally. In other words, we base change to $k^\sep[[t]]$ instead of $k[t]$, where $k^\sep$ is some separable closure of $k$. So let us first prove that the scrollar invariants behave well with respect to this passage.

\begin{lemma}
Let $\lambda$ be a partition of $d$. Then the scrollar invariants of $\lambda$ in $K\otimes_k k^\sep$ are the same as those of $\lambda$ in $K$.
\end{lemma}

\begin{proof}
The morphism $k[t]\to k^\sep[t]$ is a direct limit of étale morphisms, so by~\cite[Lem.\,10.147.5]{stacks-project} the natural map $\OO_K\otimes k^\sep[t] \to \OO_{K\otimes k^\sep[t]}$ is an isomorphism of $k^\sep[t]$-algebras. By Theorem~\ref{thm:Sdclosure.basechange} we therefore also have that
\[
\OO_L\otimes k^\sep[t] \cong G(\OO_{K\otimes k^\sep[t]}/k^\sep[t])
\]
is an isomorphism of $k^\sep[t]$-algebras. Now Theorem~\ref{thm:scrollar.invs.resolvent} for the maximal resolvent with respect to $H=1$ gives the desired result.
\end{proof}

So for the rest of this subsection, we assume that $k$ is separably closed. We denote by $\hat{K}, \hat{L}, \hat{\OO}_K, \hat{\OO}_L, ...$ the corresponding objects completed at $0$, i.e.\ tensored with $k((t))$ or $k[[t]]$. Since $S_d$-closure commutes with base change, we have that 
\[
\hat{\OO}_L = G(\hat{\OO}_K/k[[t]]).
\]
Assume that $0$ has ramification type $(e_1, \ldots, e_s)$ in $K$. Then since we are working in characteristic zero or larger than $d$, $\hat{K}$ is isomorphic to $k((t^{1/e_1}))\oplus \ldots \oplus k((t^{1/e_s}))$. Hence we also have that $\hat{\OO}_K = k[[t^{1/e_1}]]\oplus \ldots \oplus k[[t^{1/e_s}]]$ as a $k[[t]]$-extension, and so as a $k[[t]]$-algebra
\[
\hat{\OO}_L = \bigotimes_i G(k[[t^{1/e_i}]] / k[[t]]),
\]
by Theorems~\ref{thm:Sdclosure.monogenic} and~\ref{thm:Sdclosure.product}. Now note that by base change, $G(k[[t^{1/e_i}]] / k[[t]])$ is also equal to the completion of $G(k[t^{1/e_i}]/k[t])$. For the rest of this section, we fix an integer $e$ and study the $S_e$-closure $R_e = G(k[t^{1/e}] / k[t])$. Denote by $\alpha_1, \ldots, \alpha_e$ the elements $t_1^{1/e}, \ldots, t_e^{1/e}$ in $R_e$. So $\alpha_i^e = t$ and $\alpha_1\alpha_2\cdots \alpha_e = \pm t$. By Theorem~\ref{thm:Sdclosure.monogenic} $R_e$ is a ring of rank $e!$ with a basis given by the elements
\[
\alpha_1^{\ell_1}\alpha_2^{\ell_2}\cdots \alpha_e^{\ell_e}
\]
for $0\leq \ell_i < i$. Denote by $s_1, \ldots, s_e$ the elementary symmetric polynomials in the $\alpha_i$. Since the minimal polynomial of $\alpha_i$ is $X^e - t$, we have that $s_1 = \ldots = s_{e-1} = 0$ while $s_e = \pm t$. For $n = (n_1, \ldots, n_e)$ a sequence of positive integers with $n_1\geq ... \geq n_e$, we define $s_n = \prod_i s_{n_i}$. Note that $s_n = 0$ as soon as $e\nmid n_i$ for some $i$. Similarly, we denote by $m_n$ the \emph{monomial symmetric polynomial} in the $\alpha_i$ associated to $n$. In more detail this is
\[
m_n = \sum \alpha_{i_1}^{n_1} \alpha_{i_2}^{n_2} \ldots \alpha_{i_e}^{n_e},
\]
where the sum is over all conjugates of this element. For example, $m_{(2)} = \sum_i \alpha_i^2$.

\begin{lemma}\label{lem:trace.product.alphai}
Let $(n_1, \ldots, n_e)$ be a sequence of positive integers with $n_1\geq ... \geq n_e$. Then
\[
\Tr(\alpha_1^{n_1}\cdots \alpha_e^{n_e}) = \begin{cases} 0 & \text{ if } e\nmid \sum n_i, \\
c t^{\sum n_i/e} & \text{ if } e\mid \sum n_i, \end{cases}
\]
where $c\in \ZZ$.
\end{lemma}
\begin{proof}
We have that $\Tr(\alpha_1^{n_1}\cdots \alpha_e^{n_e}) = c m_n$ for some $c\in \ZZ$. It is well known that the elementary symmetric functions form a basis for the space of symmetric functions, so we may write
\[
m_n = \sum_{n'} c_{n'} s_{n'}
\]
where the sum is over decreasing sequences $n' = (n_1', \ldots, n_e')$ with $\sum n_i' = \sum n_i$. If $\sum n_i$ is not divisible by $e$ then for every $n'$ in the above sum, there is some $i$ for which $n_i'$ is not divisible by $e$ and so $m_n = 0$. If $e$ divides $\sum n_i$, then for every $n'$ appearing in the sum either $s_{n'} = 0$, or $s_{n'}$ is equal to $c_{n'} t^{\sum n_i/e}$ for some $c_{n'}\in \ZZ$. This proves the lemma.
\end{proof}

By base change, we have that $R_e\otimes k(t) = G(k(t^{1/e}) / k(t))$, which is isomorphic as an $S_e$-representation to the regular representation $k(t)[S_e]$. For $\lambda$ a partition of $e$, let $R_{e, \lambda}$ be the $\lambda$-isotypic component of $R_e$, i.e.\ it consists of all elements $x\in R_e$ which generate the representation $V_\lambda$ in $R_e\otimes k(t)$ (together with $0$). We will look for a convenient basis of $R_{e,\lambda}$, which will be based on the higher Specht polynomials from Section~\ref{sec:higher.specht}. Let $\Rco = k(t)[x_1, \ldots, x_e]/(\tilde{s}_1, \ldots, \tilde{s}_e)$ be the ring of coinvariants with its natural $S_e$-action permuting the variables, where the $\tilde{s}_i$ are the elementary symmetric polynomials in the $x_i$. For $T, S\in \ST(\lambda)$ denote by $F_T^S\in \Rco$ the higher Specht polynomial corresponding to the pair of tableaux $(T,S)$. We denote by $\tilde{F}_T^S\in k(t)[x_1, \ldots, x_e]$ the natural lift of $F_T^S$ to $k(t)[x_1, \ldots, x_n]$, which is homogeneous of degree $\ch(S)$. Let $G_T^S$ be the element of $R_e$ obtained from $F_T^S$ by replacing each $x_i$ by $\alpha_i$. While the $G_T^S$ are not quite the basis we require, we still have the following.

\begin{lemma}
With notation as above, we have the following.
\begin{enumerate}
\item For every $\lambda\vdash e$ and $S\in \ST(\lambda)$ the set
\[
\{G_T^S\mid T\in\ST(\lambda)\}
\]
generates a representation in $R_e\otimes k(t)$ which contains a unique irreducible component isomorphic to $V_\lambda$.
\item The set 
\[
\{G_T^S \mid T,S\in\ST(\lambda), \lambda\vdash e\}
\] 
is a $k(t)$-basis of $R_e\otimes k(t)$.
\end{enumerate}
\end{lemma}

\begin{proof}
For (1), let $\sigma\in S_d$. Then by Theorem~\ref{thm:higher.specht.poly} we have
\[
\sigma F_T^S = \sum_{T'\in\ST(\lambda)} a_{T'} F_{T'}^S,
\]
for some $a_{T'}$ in $k(t)$. Hence we have that in $k(t)[x_1, \ldots, x_e]$
\[
\sigma \tilde{F}_T^S = \sum_{T'\in\ST(\lambda)} a_{T'} \tilde{F}_{T'}^S + \sum_{i=1}^e s_i H_i(x_1, \ldots, x_e),
\]
where the $H_i\in k(t)[x_1, \ldots, x_e]$ are homogeneous of degree $\ch(S) - i = \deg(F_T^S) - i$. In $R_e\otimes k(t)$ this gives that
\[
\sigma G_T^S = \sum_{T'\in\ST(\lambda)} a_{T'} G_{T'}^S \pm t H_e(\alpha_1, \ldots, \alpha_e).
\]
In the quotient (as a vector space) of $R_e\otimes k(t)$ by all polynomials in the $\alpha_i$ of degree less than $ \ch(S)$ we see that $G_T^S$ generate a representation isomorphic to $V_\lambda$, by the theory of higher Specht polynomials. So acting with $S_e$ on $\{G_T^S\mid T\in\ST(\lambda)\}$ gives a representation in $R_e\otimes k(t)$ which contains a unique irreducible component isomorphic to $V_\lambda$.

For (2), we note that the irreducible components $V_\lambda$ obtained in the previous paragraph are disjoint for different $S$. In more detail, if $S, S'$ are different elements in $\ST(\lambda)$, then the two irreducible copies of $V_\lambda$ in $R_e\otimes k(t)$ obtained by acting with $S_e$ on $\{G_T^S \mid T\in \ST(\lambda)\}$ and $\{G_T^{S'} \mid T\in \ST(\lambda)\}$ have only $0$ in common. Since both $R_e\otimes k(t)$ and $\Rco$ are isomorphic as $S_e$-representations to the regular representation, this implies that the map $\Rco\to R_e\otimes k(t)$ sending $F_T^S$ to $G_T^S$ is an isomorphism of $k(t)$-vector spaces.
\end{proof}

Consider the $k(t)$-linear map
\[
\psi: \Rco\to R_e\otimes k(t): F_T^S\mapsto G_T^S.
\]
By the previous lemma, this is an isomorphism of vector spaces. We modify $\psi$ to
\[
\psi_0: \Rco\to R_e\otimes k(t): \beta \mapsto \frac{1}{e!} \sum_{\sigma\in S_e} \sigma^{-1} \phi(\sigma\beta).
\]
Then by Schur's lemma $\psi_0$ is an $S_e$-equivariant map. Moreover, by part (i) of the previous lemma $\psi_0$ is an isomorphism of $S_e$-representations. In particular, for every $\lambda\vdash e$ and $S\in \ST(\lambda)$, 
\[
\operatorname{span}\{\psi_0(F_T^S) \mid T\in\ST(\lambda)\}
\]
is an irreducible subrepresentation of $R_e\otimes k(t)$ isomorphic to $V_\lambda$. Even more, these give the decomposition of $R_e\otimes k(t)$ into its irreducible subrepresentations.

Our next goal is to prove that the elements $\psi_0(F_T^S)$ actually form a $k[t]$-basis of $R_e$.

The discriminant of $R_e$ over $k[t]$ is defined as usual as the determinant of the $e!\times e!$ matrix $(\Tr (\omega_i \omega_j))_{i,j}$ where $\{\omega_i\}_i$ is a $k[t]$-basis of $R_e$. By a small adaptation of Theorem~\ref{thm:scrollar.invs.resolvent} to this setting, we may take a $k[t]$-basis of the form $\{\omega_{\lambda, i}\}_{\lambda\vdash e}$ such that $\omega_{\lambda,i}$ generates a copy of $V_\lambda$ in $R_e\otimes k(t)$. If $\lambda\neq \lambda'$ are two partitions of $e$, then for any $i,j$, $\Tr(\omega_{\lambda,i}\omega_{\lambda', j}) = 0$. Indeed, this follows from the fact that $V_\lambda\otimes V_{\lambda'}$ does not contain the trivial representation. Hence the discriminant $\disc(R_e/k[t])$ splits into a product of discriminants $\disc(R_{e,\lambda}/k[t])$, one for each partition $\lambda\vdash e$ and which is defined as the determinant of the matrix $(\Tr(\omega_{\lambda, i}\omega_{\lambda, j})_{i,j}$ where $\{\omega_{\lambda, i}\}_i$ is any $k[t]$-basis for $R_{e,\lambda}$. We will consider discriminants only up to a unit. To prove the volume formula, the following is the key lemma. Recall that $p(\lambda) = \frac{1}{2}(\dim V_\lambda - \chi_\lambda((12)))$.

\begin{lemma}\label{lem:basis.Se.closure.k[t1/e]}
Let $\lambda\vdash e$. 
\begin{enumerate}
\item The set $\{\psi_0(F_T^S) \mid T, S\in \ST(\lambda)\}$ is a $k[t]$-basis for $R_{e, \lambda}$.
\item We have that
\[
\disc (R_{e,\lambda} / k[t]) = t^{p(\lambda)\dim(V_\lambda) (e-1) }.
\]
\end{enumerate}
\end{lemma}

\begin{proof}
We will prove both results at the same time.

First, by~\cite[Sec.\,6]{bhargavasatriano} the discriminant $\disc(R_e/k[t])$ is equal to $t^{(e-1)e!/2}$. Note also that every element of the form $\psi_0(F_T^S)$ is certainly contained in $R_{e, \lambda}$. We now consider the $S_e$-closure above infinity, $R_{e,\infty} = G(k[t^{-1/e}]/k[t^{-1}])$. The rings $R_e$ and $R_{e,\infty}$ both embed in 
\[
R_e\otimes_{k[t]} k(t) = R_{e,\infty}\otimes_{k[t^{-1}]} k(t^{-1}) = G(k(t^{1/e})/k[t]).
\]
For $T, S\in \ST(\lambda)$, note that $\psi_0(F_T^S)$ is a $\ZZ$-linear combination of conjugates of the element $\alpha_1^{s(S,1)}\cdots \alpha_d^{s(S,d)}$. Therefore, we have that $t^{-m(S)}\psi_0(F_T^S)$ is an element of $R_{e,\infty}$, where $m(S)$ is the largest index in the charge word of $S$. Since the discriminant of $R_{e,\infty,\lambda}$ is an element of $k[t^{-1}]$, we see that the degree of the discriminant of $R_{e,\lambda}$ is bounded by 
\[
\dim V_\lambda \cdot \sum_{S\in\ST(\lambda)} m(S) = p(\lambda)(e-1)\dim(V_\lambda),
\]
by Corollary~\ref{cor:sum.maximal.reading.word}. Using that
\[
\sum_{\lambda\vdash e} p(\lambda)(e-1)\dim(V_\lambda) = (e-1)\frac{e!}{2},
\]
and the fact that $R_e$ has discriminant $t^{(e-1)e!/2}$ over $k[t]$, we obtain that
\[
\disc(R_{e,\lambda}/k[t]) = (t^{p(\lambda)(e-1)\dim V_\lambda}),
\]
and that the elements $\{\psi_0(F_T^S) \mid \lambda\vdash e, T,S\in\ST(\lambda)\}$ form a $k[t]$-basis for $R_e$.
\end{proof}

\subsection{The volume formula}

To prove Theorem~\ref{thm:schreyer.is.scrollar}, we will rely on a formula for the sum of the scrollar invariants of an irreducible representation. So for $\lambda$ a partition of $d$, we define the \emph{volume of $\lambda$ (with respect to $K$)} to be
\[
\vol_K (\lambda) = e_{\lambda, 1} + \ldots + e_{\lambda, \dim V_\lambda},
\]
i.e.\ it is the sum of the scrollar invariants of $\lambda$ with respect to $L$. The terminology comes from the analogy with the successive minima of a number field and Minkowski's second theorem~\cite[Sec.\,7]{hessRR}.

\begin{lemma}\label{lem:discr.formula.lambda.part}
Let $\lambda$ be a partition of $d$. Then
\[
\disc \OO_\lambda = \left( \disc \OO_K \right)^{p(\lambda) \dim V_\lambda},
\]
and similarly for $\OO_{\lambda, \infty}$.
\end{lemma}

\begin{proof}
The statement is local, so it is enough to prove that 
\[
\disc \hat{\OO}_\lambda = \left( \disc \hat{\OO}_K \right)^{p(\lambda) \dim V_\lambda}.
\]
Assume that $0$ has ramification pattern $(e_1, \ldots, e_r)$ in $C\to \PP^1$. As mentioned above, we have that
\[
\hat{\OO}_L = \Ind_{S_{e_1}\times\ldots \times S_{e_r}}^{S_d}\left(\bigotimes_i G(k[t^{1/e_i}]/k[t])\otimes_{k[t]} k[[t]]\right).
\]
As in Section~\ref{sec:Se.closure.k[t1/e]}, denote $G(k[t^{1/e_i}]/k[t])$ by $R_{e_i}$. Since the discriminant passes through the completion, it is enough to prove that 
\[
\disc\left( \Ind_{S_{e_1}\times\ldots \times S_{e_r}}^{S_d}\big( \bigotimes_i R_{e_i} \big)_\lambda \right)  = \left(t^{p(\lambda)\dim V_\lambda \sum_i (e_i-1)}\right).
\]
To do this, we construct a $k[t]$-basis for the module on the left hand side. We can decompose 
\[
\Res_{S_{e_1}\times \ldots \times S_{e_r}}^{S_d} V_\lambda = \bigoplus_{(\mu_1, \ldots, \mu_r)} a_\mu^\lambda V_{\mu_1}\otimes \ldots \otimes V_{\mu_r},
\] 
where the sum is over all partitions $\mu_1, \ldots, \mu_r$ of $e_1, \ldots, e_r$, and $a_\mu^\lambda\in \ZZ$. Note that then for every $i$
\[
\chi_\lambda = \sum_{\mu_i\vdash e_i} \left( \sum_{\mu_1, \ldots, \hat{\mu}_i, \ldots, \mu_r} a_\mu^\lambda \prod_{j\neq i} \dim V_{\mu_j}\right) \cdot \chi_{\mu_i}.
\] 
For any $\mu_i\vdash e_i$ we fix a standard Young tableau $T_{\mu_i}$ and define the set 
\[
\{\omega_{\mu_i, 1}, \ldots, \omega_{\mu_i, \dim V_{\mu_i}}\} = \{\psi_0(F_T^S) \mid S\in \ST(\mu_i)\}.
\]
By Lemma~\ref{lem:basis.Se.closure.k[t1/e]}, acting with $S_{e_i}$ on $\{\omega_{\mu_i, j}\}_j$ gives a $k[t]$-basis for $R_{e_i, \mu_i}$. Let $\mu_1, \ldots, \mu_r$ be partitions of $e_1, \ldots, e_r$, and let $\ell_i\in \{1, \ldots, \dim V_{\mu_i}\}$. For each element 
\[
\omega(\mu, \ell) = \omega_{\mu_1, \ell_1}\otimes \ldots \otimes \omega_{\mu_r, \ell_r}\in \bigotimes_i R_{e_i, \mu_i}
\]
we obtain $a_{\mu}^\lambda$ elements in 
\[
\Ind_{S_{e_1}\times\ldots \times S_{e_r}}^{S_d}\big( \bigotimes_i R_{e_i} \big)_\lambda
\] 
each of which generates a different copy of $V_\lambda$ upon tensoring with $k(t)$. Together, these elements will form a $k[t]$-basis for $\Ind_{S_{e_1}\times\ldots \times S_{e_r}}^{S_d}\big( \bigotimes_i R_{e_i} \big)_\lambda$. When we compute the discriminant we obtain
\begin{align*}
\disc \left( \Ind_{S_{e_1}\times\ldots \times S_{e_r}}^{S_d}\big( \bigotimes_i R_{e_i} \big)_\lambda \right) &= t^{\dim V_\lambda \sum_\mu \sum_{S_i\in \ST(\mu_i)} a_\mu^\lambda \sum_i m(S_i)} \\
&= t^{\dim V_\lambda \sum_i \sum_{\mu_1, \ldots, \hat{\mu}_i, \ldots, \mu_r} \sum_{\mu_i} a_\mu^\lambda p(\mu_i)(e_i-1)\prod_{j\neq i} \dim V_{\mu_j}} \\
&= t^{\dim V_\lambda p(\lambda) \sum_i (e_i-1)},
\end{align*}
where we have used Corollary~\ref{cor:sum.maximal.reading.word}. Since the characteristic of $k$ does not divide $d$, the extension $\OO_K/k[t]$ is tamely ramified and we have that $\disc(\OO_K/k[t]) = t^{\sum (e_i-1)}$. Hence we conclude that indeed
\[
\disc(\OO_\lambda/k[t]) = \disc(\OO_K/k[t])^{\dim (V_\lambda) p(\lambda)}. \qedhere
\]
\end{proof}

We can deduce our volume formula.

\begin{corollary}[Volume formula]\label{cor:volume.formula}
Let $\lambda$ be a partition of $d$. Then
\[
\vol_K(\lambda) = p(\lambda)\cdot (g+d-1).
\]
\end{corollary}

\begin{proof}
The representation $V_{(1^d)}$ is one-dimensional. Let $\delta\in \OO_L$ be a reduced basis of $\OO_{(1^d)}$ and note that $\delta^2\in k[t]$. Then by Lemma \ref{lem:discr.formula.lambda.part}, we have that $\disc (\OO_K/k[t]) = \delta^2$. Let $\omega_1, \ldots, \omega_r$ be a reduced basis for $\OO_\lambda$, where $r = \dim (V_\lambda)^2$, with scrollar invariants $e_{\lambda, 1}, \ldots, e_{\lambda, r}$. Consider the discriminant matrix $T = (\Tr(\omega_i\omega_j))_{i,j}$, which has determinant 
\[
\disc(\OO_\lambda/k[t]) = \disc(\OO_K/k[t])^{p(\lambda)\dim (V_\lambda)} = \delta^{2p(\lambda)\dim(V_\lambda)},
\]
by Lemma~\ref{lem:discr.formula.lambda.part}. We now play the same game over infinity, so we define $T_\infty = (\Tr(t^{-e_{\lambda, i}-e_{\lambda, j} } \omega_i\omega_j)))_{i,j}$. On the one hand, by simply taking out all powers of $t$ from $T_\infty$, we have that
\[
\frac{\det T}{\det T_\infty} = t^{\dim (V_\lambda)(2 \vol_K \lambda )}.
\]
On the other hand, the discussion above shows that
\begin{align*}
\frac{\det T}{\det T_\infty} &= \frac{\disc (\OO_\lambda/k[t])^{p(\lambda)\dim (V_\lambda)}}{\disc (\OO_{\infty, \lambda}/k[t^{-1}])^{p(\lambda)\dim (V_\lambda)}} = \frac{\delta^{2p(\lambda)\dim (V_\lambda)}}{(t^{-g-d+1}\delta)^{2p(\lambda)\dim (V_\lambda)}} \\
 &= t^{(2g+2d-2)p(\lambda)\dim(V_\lambda)},
\end{align*}
from which the result follows.
\end{proof}

As for resolvent curves, we obtain the following.

\begin{corollary}\label{lem:discr.formula.resolvents}
Let $H$ be a subgroup of $S_d$. Then
\[
\disc \OO_L^H = \left( \disc \OO_K \right)^{p(H)},
\]
and similarly for $\OO_{L, \infty}^H$, where 
\[
p(H) = \frac{1}{2}\left( \dim \Ind_H^{S_d} \mathbf{1} -  \chi_{\Ind_H^{S_d} \mathbf{1} }((12))\right).
\]
\end{corollary}

Equivalently, one may define $p(H)$ as 
\[
p(H) = (d-2)! \frac{\# \{\text{transpositions } \sigma\notin H\} }{\# H}.
\]

\begin{proof}
This follows from Lemma~\ref{lem:discr.formula.lambda.part} and Theorem~\ref{thm:scrollar.invs.resolvent}.
\end{proof}

One may also state this result as a computation of the (arithmetic) genus of the resolvent curve $\res_H C$.

\begin{corollary}\label{cor:genus.resolvent}
Let $H$ be a subgroup of $C$ and suppose that $\res_H C$ is irreducible. Then the (arithmetic) genus of $\res_H C$ is equal to
\[
p(H)(g+d-1) + 1 - [S_d:H].
\]
\end{corollary}

\begin{proof}
This follows from the previous result via the Riemann--Hurwitz formula~\cite[Prop.\,III.3.13]{neukirch}.
\end{proof}

\subsection{Duality}

There is a duality statement connecting the scrollar invariants of a partition $\lambda$ and those of the dual partition $\lambda^*$, obtained by transposing the corresponding Young diagram. Recall that we have that $V_{\lambda^*} = V_\lambda \otimes V_{(1^d)}$.

\begin{proposition}[Duality]\label{prop:duality}
Let $\lambda$ be a partition of $d$ with scrollar invariants $\{e_{\lambda, i}\}_i$ in $L$. Then the scrollar invariants of $\lambda^*$ are given by $\{g+d-1-e_{\lambda, i}\}_i$.
\end{proposition}
\begin{proof}
Let $\omega_1, \ldots, \omega_r$ be a reduced basis for $\OO_\lambda$, where $r=\dim (V_\lambda)^2$, with scrollar invariants $e_{\lambda, 1}, \ldots, e_{\lambda, r}$. Let $\delta$ be a reduced basis of the one-dimensional $\OO_{(1^d)}$. By Lemma \ref{lem:discr.formula.lambda.part} the scrollar invariant of $\delta$ is $g+d-1$. Since the extension $K/k(t)$ is separable, the trace pairing $\Tr: L\to k(t)$ is non-degenerate. Hence there exists a dual basis $\omega_1^*, \ldots, \omega_r^*$ to $\omega_1, \ldots, \omega_r$ with the property that $\Tr(\omega_i\omega_j^*)$ is $1$ if $i=j$ and $0$ if $i\neq j$. We claim that the elements
\[
\delta \omega_i^*, i=1, \ldots, \dim(V_{\lambda})^2
\]
form a reduced basis for $\OO_{\lambda^*}$. These elements are indeed contained in $W_{\lambda^*}$. The strategy for this will be to use Lemma~\ref{lem:candidate.basis.is.reduced}. First, note that the candidate scrollar invariants indeed sum up to the correct value, since
\[
\sum_i (g+d-1-e_i)  = \frac{1}{2} (\dim (V_\lambda) + \chi_{\lambda}((12)) )(g+d-1).
\]
So we want to prove that $\delta\omega_i^*$ is in $\OO_{\lambda^*}$ and that $t^{-(g+d-1-e_{\lambda, i}) }\delta\omega_i^*$ is in $\OO_{\infty, \lambda^*}$. 

To prove the claim, we work locally again. Indeed, $\delta\omega_i^*$ is in $\OO_{\lambda^*}$ if and only if it is contained in $\hat{\OO}_{\lambda^*}$ for every completion of $L$. So let us localize at $0$. Assume that $0$ ramifies as $(e_1, \ldots, e_s)$ in $C\to \PP^1$. Then 
\[
\hat{\OO}_L = \bigotimes_i G(k[[t^{1/e_i}]] / k[[t]]).
\]
We first focus on one of the components $G(k[[t^{1/e_i}]] / k[[t]])$. For these we have an explicit basis via higher Specht polynomials by Lemma~\ref{lem:basis.Se.closure.k[t1/e]}. Let $\mu$ be a partition of $e_i$. Let $T, S\in \ST(\mu)$ and $T', S'\in \ST(\mu^*)$. Then $\psi_0(F_T^S)\psi_0(F_{T'}^{S'})$ is in $\OO_L$, and so the coefficient of this product at $\delta$ is in $k[t]$. This coefficient is (up to an element of $k^\times$) $\Tr(\psi_0(F_T^S)\psi_0(F^{T'}_{S'})/\delta)$. Since $\ch(S) + \ch(S') < 2\binom{e}{2}$ if $\mu\neq (1^e)$ we have in fact that $\Tr(\psi_0(F_T^S)\psi_0(F^{T'}_{S'})/\delta)$ is in $k$. Since the matrix with these elements as its entries is invertible (as the $\psi_0(F_T^S)/\delta$ generate $W_{\lambda^*}$ as a representation), we conclude that the $\delta\psi_0(F_T^S)^*$ are in $G(k[[t^{1/e_i}]] / k[[t]])$. A similar strategy to the proof of Lemma~\ref{lem:discr.formula.lambda.part} allow us to conclude that $\delta \psi_0(F_T^S)^*$ is in $\hat{\OO}_L$. The same reasoning above the other points of $\PP^1$ give the desired conclusion.
\end{proof}

\subsection{Scrollar invariants of hooks}

As in the simply branched case, the scrollar invariants of hooks $(d-i, 1^i)$ have a simple description in terms of the usual scrollar invariants $e_i$ of $C\to \PP^1$. 

\begin{proposition}\label{prop:hooks}
Let $i\in \{0, \ldots, d-1\}$. The scrollar invariants of the partition $(d-i, 1^i)$ with respect to $L$ are 
\[
\left\{ \sum_{\ell\in S} e_\ell \mid S\subset \{1, \ldots, d-1\}, \#S = i \right\}.
\]
\end{proposition}

\begin{proof}
This is identical to \cite[Prop.\,1]{syzrep}.
\end{proof}

\subsection{Maximal resolvents}

Compared to~\cite{syzrep}, our resolvent curves can become singular or reducible. So it is a natural question to classify precisely when this happens. In this section we prove Proposition~\ref{prop:resolvent.maximal}. This shows that $\res_H C$ being singular is determined representation theoretically, and depends only on the ramification of $C\to \PP^1$.

For the proof, we first rewrite what it means for $\res_H C$ to be maximal in terms of characters of $S_d$.

Let $\tilde{\OO}_L$ be the maximal order of $L$, i.e.\ it is the integral closure of $k[t]$ in $L$. Similarly define $\tilde{\OO}_{L, \infty}$ as the integral closure of $k[t^{-1}]$ in $L$. For $\lambda\vdash d$ we define $\tilde{\OO}_\lambda$ as $\tilde{\OO}_L \cap W_\lambda$, where $W_\lambda$ is the isotypic component of $\lambda$ in $L$. Let us say that $\lambda$ is \emph{maximal in $L$} if $\OO_\lambda = \tilde{\OO}_\lambda$. For $p$ a point in $\PP^1$, let us say that $\lambda$ is \emph{maximal in $L$ at $p$} if $\OO_\lambda$ is equal to $\tilde{\OO}_\lambda$ after localizing at $p$. By Theorem~\ref{thm:scrollar.invs.resolvent}, we see that $\OO_{L}^H$ is maximal if and only if for every partition $\lambda$ of $d$ such that $\mult(V_\lambda, \Ind_H^{S_d}\lambda)>0$ we have that $\lambda$ is maximal in $L$.

For a partition $e = (e_1, \ldots, e_r)$ of $d$ let $\rho\in S_d$ be a permutation with cycle type $(e_1, \ldots, e_r)$. For an integer $i$ we let $e^i$ be the partition of $d$ corresponding to the cycle type of the permutation $\rho^i$. We denote by $\lcm(e) = \lcm\{e_1, \ldots, e_r\}$. We begin by computing the discriminant of the maximal order of the resolvent.

\begin{lemma}\label{lem:disc.maximal.order.resolvent}
Let $H$ be a subgroup of $S_d$. Assume that $0$ ramifies as $e = (e_1, \ldots, e_r)$ in $\OO_K/k[t]$. Then the discriminant of the completion $\hat{\tilde{\OO}}_L^H$ of $\tilde{\OO}_L^H$ above $0$ is equal to
\[
\prod_{i=1}^{\lcm(e) -1} t^{\frac{n(H, e^i)}{\lcm(e)}},
\]
where for $\sigma\in S_d$, $n(H, \sigma) = \chi_{\Ind_H^{S_d}\mathbf{1}}(\id) - \chi_{\Ind_H^{S_d}\mathbf{1}}(\sigma)$.
\end{lemma}

\begin{proof}
We rely on a discriminant formula due to Lenstra, Pila and Pomerance, to compute the discriminant of $\tilde{\OO}_L^H$~\cite[Thm.\,4.4]{hyperelliptic}. This result is only stated for number fields, but the formula carries over directly to our setting, and states that
\[
(\disc \tilde{\OO}_L^H)^{d!} = \prod_{\sigma\in S_d\setminus\{\id\}} \Norm_{L/k(t)} (\II_\sigma)^{\# \{\tau: L^H\hookrightarrow L \mid \sigma \circ \tau \neq \tau\}},
\]
where $\II_\sigma$ is the ideal of $\tilde{\OO}_L$ generated by all elements of the form $\sigma(x)-x$, for $x\in \tilde{\OO}_L$. Note that a non-zero prime $\mathcal{B}\subset \tilde{\OO}_L$ above $0$ divides $\II_\sigma$ if and only if $\sigma$ is in the inertia group of $\mathcal{B}$. By~\cite[SZ.\,1]{vanderwaerden} or~\cite[Cor.\,7.59]{milne_ant}, this inertia group is cyclic and generated by a permutation $\rho$ of cycle type $(e_1, \ldots, e_r)$. So in the product we can disregard those $\sigma$ which are not conjugate to a power of $\rho$. Now note that
\[
\# \{\tau: L^H\hookrightarrow L \mid \sigma \circ \tau \neq \tau\} = n(H, \sigma).
\]
Since $\charac k$ does not divide $d$, the extension $K/k(t)$ is tamely ramified, ensuring that $\II_\sigma$ is a square-free ideal, i.e.\ $\BB^2$ never divides $\II_\sigma$. This description of the inertia group moreover yields that the ramification index of $0$ in $\OO_L$ is equal to $\lcm (e)$, the least common multiple of the ramification indices of $0$ in $\OO_K$. Hence the number of primes above $0$ in $\tilde{\OO}_L$ is equal to $d! / \lcm (e)$. Since the action of $S_d$ is transitive on the primes above $0$ in $\tilde{\OO}_L$, a computation shows that
\[
\disc(\hat{\tilde{\OO}}_L^H/k[[t]]) = \prod_{i=1}^{\lcm(e)-1} t^{\frac{n(H, e^i)}{\lcm (e)}},
\]
as desired.
\end{proof}

\begin{lemma}\label{lem:resolvent.maximal.lem}
Let $\lambda$ be a non-trivial partition of $d$. Then $\lambda$ is maximal in $L$ if and only if for every ramification pattern $e = (e_1, \ldots, e_r)$ of $C\to \PP^1$ and every $i=1, \ldots, \ord e - 1$, we have that
\[
\frac{\dim V_\lambda - \chi_\lambda(e^i)}{\dim V_\lambda - \chi_\lambda((12))} = \frac{d-\#\mathrm{Fix}(e^i)}{2},
\]
where $\mathrm{Fix}(e^i)\subset \{1, \ldots, d\}$ is the set of fixed points of a permutation of cycle type $e^i$.
\end{lemma}

\begin{proof}
The statement is local and geometric, so it is enough to prove this above $0\in \PP^1$ with the assumption that $k$ is algebraically closed. Let $e = (e_1, \ldots, e_r)$ be the ramification pattern of $0$ in $C\to \PP^1$. We first investigate when $\OO_L^H$ is maximal. By Lemma \ref{lem:discr.formula.resolvents} the resolvent $\OO_L^H$ has discriminant
\[
(\disc \OO_K)^{p(H)}.
\] 
Note that $\OO_L^H$ is maximal if and only if its discriminant is equal to the discriminant of $\tilde{\OO}_L^H$. So by the previous lemma for $H$ and $S_{d-1}$, $\OO_L^H$ is maximal if and only if for every $i=1, \ldots, \ord e-1$ we have that
\[
n(H, e^i) = p(H)n(S_{d-1}, e^i).
\]
Rewriting, this equation is equivalent to the statement that
\[
\frac{[S_d:H]- \chi_{\Ind_H^{S_d}\mathbf{1}}(e^i)}{[S_d:H]- \chi_{\Ind_H^{S_d}\mathbf{1}}((12))} = \frac{d-\#\mathrm{Fix}(e^i)}{2}.
\]
The result now follows by considering all Young subgroups $S_\lambda$ in $S_d$, for all partitions $\lambda$ of $d$.
\end{proof}

\begin{proof}[Proof of Proposition~\ref{prop:resolvent.maximal}]
The proof follows directly from Lemma~\ref{lem:repr.lemma.for.maximality}, Lemma~\ref{lem:resolvent.maximal.lem} and the discussion above.
\end{proof}

\subsection{Irreducible resolvents}

If the Galois closure of $K/k(t)$ does not have the full symmetric group $S_d$ as its Galois group, then the curve corresponding to $L$ will be reducible. However, it can still happen that a resolvent $L^H$ is irreducible. This will depend only on the Galois group of the Galois closure of $K/k(t)$ and the subgroup $H$, as described by Proposition~\ref{prop:resolvent.irre}.

\begin{proof}[Proof of Proposition~\ref{prop:resolvent.irre}]
Denote by $M$ the Galois closure of the extension $K/k(t)$. Let $r = [S_d:G] =  d! / \#G$. Then by \cite{bhargavasatriano}, $L$ is isomorphic to $M^r$ as a $k(t)$-algebra. In fact, $L$ is isomorphic as a $k(t)[S_d]$-algebra to
\[
\Ind_{G}^{S_d} M.
\]
Let us be a bit more explicit about the $S_d$-action on $L$. Let $s_1, \ldots, s_r$ be coset representatives for $G$ in $S_d$. Then we can write $L = \oplus_{i=1}^r s_i M$ as $M[S_d]$-module. Every $\sigma\in S_d$ determines a permutation $j_\sigma$ of the $s_i$ via $\sigma s_i = s_{j_\sigma(i)} g_{\sigma, i}$ for some $g_{\sigma, i}\in G$. For $x = \sum_i s_i x_i$ in $L$ the action of $\sigma$ on $x$ is defined by
\[
\sigma(x) = \sigma\left( \sum_i s_i x_i\right) = \sum_i s_{j_\sigma(i)} g_{\sigma, i}(x_i).
\]
Now assume that $GH=S_d$. Equivalently, this means that the action of $H$ on $S_d/G$ via left translation is transitive. Take any element $x = \sum_i s_i x_i$ in $L^H$. For every $i=1, \ldots, r$ there is by transitivity a $\sigma \in H$ such that $j_{\sigma}(1) = i$. Acting with this $\sigma$ on $x$ shows that $v_i = g_{\sigma, 1}(v_1)$ and hence all non-zero elements of $L^H$ are invertible. In other words, $L^H$ is a field. 

Conversely, assume that $GH\neq S_d$. Equivalently, the action of $H$ on $S_d/G$ is not transitive. Then we find $i,j\in \{1, \ldots, r\}$ such that for every $\sigma\in H$ we have that $j_{\sigma}(i)\neq j$. Now consider the element 
\[
\sum_{h\in H} h\cdot (s_i\cdot 1) \in L,
\]
(here $s_i\cdot 1$ is interpreted in $L = \oplus_{i=1}^r s_i M$.) The result is clearly a non-zero non-invertible element of $L^H$, since the coordinate at $s_j$ will be zero while the coordinate at $s_i$ is $\#H\neq 0$. Hence $L^H$ is not a field.
\end{proof}

\begin{remark}
In fact, the above proof shows somewhat more. Consider the action of $H$ on $S_d/G$, and assume that its orbits have sizes $r_1, \ldots, r_m$. Then $\res_H C$ is the union of $m$ curves $C_1, \ldots, C_m$, and for each $i$ $\res_H \varphi$ restricts to a morphism $C_i\to \PP^1$ of degree $r_i$.	
\end{remark}

\section{Scrollar syzygies and splitting types}

In this section we prove Theorem~\ref{thm:schreyer.is.scrollar}. To do this, we first recall how one can construct a minimal free resolution of $d$ points in $\PP^{d-2}$ using Galois theory~\cite{syzrep}. Afterwards, we use this resolution to study the splitting types of the Casnati--Ekedahl resolution of a degree $d$ cover $C\to \PP^1$, and relate these to our generalized scrollar invariants.

This section is completely analogous to Sections 4 and 6 from~\cite{syzrep}, but replacing the Galois closure by the $S_d$-closure. With the theory developed in the previous section, all proofs in~\cite{syzrep} carry over with minimal changes, hence no proofs will be given. 

\subsection{Scrollar syzygies from representation theory}\label{sec:resolution}

Throughout this subsection, $d\geq 4$ is a positive integer and $F$ will denote a field of characteristic zero or larger than $d$. This will later be specified to $F=k(t)$, but the construction of the free resolution works more generally. Recall that for $2\leq i\leq d-2$ we have defined the partitions $\lambda_i = (d-i, 2, 1^{i-2-1})$ of $d$. We extend this notation by $\lambda_0 = (d), \lambda_1 = (d-1, 1), \lambda_d = (1^d)$ (but we leave $\lambda_{d-1}$ undefined). 

Let $K/F$ be a degree $d\geq 4$ field extension, which is automatically separable because $F$ has characteristic $0$ or larger than $d$. Let $L = G(K/F)$ be the $S_d$-closure of $K/F$. By \cite[Thm.\,2]{bhargavasatriano} this is a ring of rank $d!$ over $F$. As in \cite{syzrep}, we will construct a minimal free resolution of a certain configuration of $d$ points in $\PP^{d-2}$ using Galois theory. All proofs are identical, since the arguments in the Galois closure work just as well in the formal $S_d$-closure. So we omit proofs, but give the construction of the resolution for later use.

Let $1 = \alpha_0, \alpha_1, \ldots,\alpha_{d-1}$ be a basis for $K/F$, where without loss of generality we can assume that $\Tr_{L/F}( \alpha_i) = 0$ for $i=1, \ldots, d-1$. Denote by $\alpha_i^*$ the dual basis of $\alpha_0, \alpha_1, \ldots, \alpha_{d-1}$ with respect to the trace pairing $\Tr_{L/F}$ on $K$, i.e.\ $\Tr_{L/F}(\alpha_i\alpha_j^*) = 1$ if $i=j$ and $0$ if $i\neq j$. Let $\rho^{(1)}, \ldots, \rho^{(d)}$ be the $d$ distinct field embeddings of $K$ in $F^\alg$, an algebraic closure of $F$. Consider the $d$ points
\[
(\rho^{(1)}(\alpha_1^{*}): \ldots : \rho^{(1)}(\alpha_{d-1}^{*})), (\rho^{(2)}(\alpha_1^{*}): \ldots : \rho^{(2)}(\alpha_{d-1}^{*})), \ldots, (\rho^{(d)}(\alpha_1^{*}): \ldots : \rho^{(d)}(\alpha_{d-1}^{*}))
\]
in $\PP^{d-2}$. These points are distinct, and no $d-1$ lie on a hyperplane. This set of points is closed under the action of $S_d$ acting on the upper indices in the $\rho^{(i)}$ and hence is defined over $F$. Let $I$ be the ideal in $F[x_1, \ldots, x_{d-1}]$ of this set of points. We will give a Galois-theoretic construction for the minimal free resolution of $R = F[x_1, \ldots, x_{d-1}]/I$.

Recall that, as an $S_d$-representation, $L$ is isomorphic to 
\[
L \cong F[S_d],
\]
i.e.\ $L$ is the regular representation of $S_d$. Hence it splits as a direct sum of subrepresentations $L = \oplus_\lambda W_\lambda$, where the sum is over all partitions $\lambda$ of $d$ and $W_\lambda$ is the isotypic component corresponding to $\lambda$. Every $W_\lambda$ is isomorphic to $V_\lambda^{\dim V_\lambda}$. For $2\leq i\leq d-2$ or $i=0,1,d$ define $V_i = W_{\lambda_i} \cap L^{S_{\lambda_i}}$ and write $\beta_{i-1} = \dim V_i$. We take a $F$-bases $\alpha_1, \ldots, \alpha_{d-1}$ of $V_1$ and $\omega_1^i, \ldots, \omega_{\beta_{i-1}}^i$ of $V_i$ for $2\leq i \leq d-2$.

Then our resolution will take the form
\begin{multline}\label{eqn:resolution.d.points}
0 \to V_d^\ast \otimes R(-d) \to V_{d-2}^\ast \otimes R(-d+2) \to V_{d-3}^\ast \otimes R(-d+3) \to \ldots \\
\to V_4^\ast \otimes R(-4) \to V_3^\ast \otimes R(-3) \to V_2^\ast \otimes R(-2) \to V_0^\ast \otimes R \to R/I \to 0,
\end{multline}
where $V_i^* = \Hom_F(V_i, F)$. The polynomial ring $R$ will come about as $\Sym V_1^*$.

\subsection{The construction}

We know explain how the maps in the resolution are constructed, beginning with the first step in the resolution. Let $y_1, \ldots, y_{d-1}$ be an $F$-basis of $V_{(d-1,1)}$ such that $y_1$ is fixed by $S_{d-1}$ and the other $y_i$ are conjugate to $y_1$ in a way compatible with the $S_d$-action. Recall that $\Sym^2 V_{(d-1,1)}$ decomposes as $V_{(d)}\oplus V_{(d-1,1)}\oplus V_{(d-2,2)}$. Hence by Lemma \ref{lem: induced and fixed space}, there is, up to scalar multiplication, a unique element 
\[
p^1 = \sum_{m,n=1}^{d-1} p^1_{mn} y_m\otimes y_n \in \Sym^2 V_{(d-1,1)}
\]
which is fixed by $S_2\times S_{d-2}$ and which generates $V_{(d-2,2)}$, where $p_{mn}^1\in F$. We may assume that $p^1_{mn} = p^1_{nm}$ for all $m,n$. Use this element to construct the following quadratic map
\[
\psi_1: \Sym^2 V_1\to V_2: \alpha\otimes \beta\mapsto \sum_{m,n=1}^{d-1} p^1_{mn} \alpha^{(m)} \beta^{(n)}.
\]
Here the $p^1_{mn}$ are in $F$. However, note that this construction works just as well over the prime subfield of $F$, so we can assume that all $p^1_{mn}$ are already contained in this prime subfield. By dualizing, we obtain a map $\psi_1^*: V_2^*\to \Sym^2 V_1^*$. We now identify $V_1^*$ with $R_1$ to obtain a linear map $\psi_1^*: V_2^*\otimes R(-2)\to V_0^*\otimes R$, which is the first step in our resolution. 

Now let $2\leq i\leq d-3$ and let $y_1, \ldots, y_{d-1}$ be as above. Take an $F$-basis $w_1^i, \ldots, w_{\beta_{i-1}}^i$ of $V_{\lambda_i}$ such that $w_1^i$ is fixed by $S_{\lambda_{i}}$ and all other $w_j^i$ are conjugate to $w_1^i$. By Lemma \ref{lem: multiplicity fixed} and Lemma \ref{lem: induced and fixed space} there is, up to scalar multiplication, a unique element
\[
p^i = \sum_{m=1}^{\beta_{i-1}}\sum_{n=1}^{d-1} p^i_{mn} w_m^i \otimes y_n \in V_{\lambda_i} \otimes V_{d-1,1}
\]
which is fixed by $S_{\lambda_{i+1}}$ and generates the representation $V_{\lambda_{i+1}}$. Again, the $p^i_{mn}$ may be taken to be in the prime subfield of $F$. Use this element to construct the map
\[
\psi_i: V_i\otimes V_1\to V_{i+1}: \omega\otimes \alpha\mapsto \sum_{m=1}^{\beta_{i-1}} \sum_{n=1}^{d-1} p^i_{mn} \omega^{(m)} \alpha^{(n)},
\]
where the conjugation of $\omega$ is labelled compatibly with $w_1^i$, i.e.\ $\omega^{(m)} = \sigma(\omega)$ for any $\sigma\in S_d$ with $\sigma(w_1^i) = w_m^i$. Upon dualizing and identifying $V_1^*$ with $R_1$, we obtain maps $\psi_i^*: V_{i+1}^*\otimes R(-i-1)\to V_i^*\otimes R(-i)$, which will be the next steps in the resolution.

For the final step in the resolution, consider the representation $V_{\lambda_{d-2}}\otimes \Sym^2 V_{(d-1,1)}$ which has a unique subrepresentation isomorphic to $V_{(1^d)}$. Then there exists an element, unique up to scalar multiplication,
\[
p^{d-2} = \sum_{i=1}^{\beta_{d-3}} \sum_{j,\ell=1}^{d-1} p_{ij\ell}^{d-2} w_i^{d-2} \otimes (y_j\otimes y_\ell)
\]
on which $S_d$ acts as the sign representation, where the $p_{ij\ell}^{d-2}$ are contained in the prime subfield of $F$. We may assume that $p_{ij\ell}^{d-2} = p_{i \ell j}^{d-2}$ for all $i,j,\ell$. As before, this gives a linear map
\[
\psi_{d-2}: V_{d-2} \otimes \Sym^2 V_1 \to V_d.
\]
Upon dualizing and identifying $V_1^*$ with $R_1$ we obtain the final map of the resolution $\psi_{d-2}^*: V_d^*\otimes R(-d)\to V_{d-2}^*\otimes R(-d+2)$.

\begin{theorem}
With the notation as above, the sequence \eqref{eqn:resolution.d.points} is a minimal free resolution of $R/I$.
\end{theorem}

\begin{proof}
As mentioned above, the proof is completely identical to the work in \cite[Sec.\,4]{syzrep}. The computations there were done with the assumption that the Galois closure of $K/F$ has Galois group $S_d$. However, all of the arguments hold just as well in the formal $S_d$-closure $L$ here \footnote{There is one point in \cite[Lem.\,14]{syzrep} where it is used that $F$ is infinite to conclude that certain polynomials vanishing everywhere are zero. However, these polynomials have degree at most $4$ and we work in characteristic $0$ or larger than $d$. Since $d$ is at least $4$, we can still conclude from vanishing everywhere that they are the zero polynomial.}. 
\end{proof}

\begin{remark}
As in~\cite[Sec.\,4.14]{syzrep}, this construction also works for $d=3$. In that case, we obtain a resolution of $3$ points in $\PP^1$ of the form
\[
0 \to V_3^*\otimes R(-3) \to V_0^*\otimes R \to R/I \to 0.
\]
\end{remark}

\subsection{The splitting types are scrollar}

In this section we give a scrollar interpretation for Schreyer's invariants. Let us first recall the setting. Let $ \varphi: C\to \PP^1$ be a degree $d\geq 4$ cover, where $C$ is a nice curve of genus $g$. Define the vector bundle $\EE$ on $\PP^1$ as
\[
\EE = \OO_{\PP^1}(e_1)\oplus \OO_{\PP^1}(e_2)\oplus \ldots \oplus \OO_{\PP^1}(e_{d-1}),
\]
where $e_1, \ldots, e_{d-1}$ are the scrollar invariants of $C\to \PP^1$. The map $\varphi$ decomposes as the composition of the \emph{relative canonical embedding} $C\hookrightarrow \PP(\EE)$ with the natural bundle map $\pi: \PP(\EE)\to \PP^1$. The Picard group of $\PP(\EE)$ is generated by two classes $R, H$ where $R$ is the \emph{ruling class} $R = \pi^*\OO_{\PP^1}(1)$, and $H$ is the \emph{hyperplane class} $H = j^* \OO_{\PP^{g+2d-3}}(1)$ with $j: \PP(\EE) \to \PP^{g+2d-3}$ the tautological map. Casnati--Ekedahl~\cite{casnati_ekedahl}, generalizing work by Schreyer~\cite{schreyer}, give a minimal graded free resolution of $C$ inside $\PP(\EE)$ of the form 
\begin{align*}
0 \rightarrow \mathcal{O}_{\PP(\mathcal{E})} (-dH + (g-d-1)R) \rightarrow 
 \\
\bigoplus_{j=1}^{\beta_{d-3}} \mathcal{O}_{\PP(\mathcal{E})}(-(d-2)H + b_j^{(d-3)}R) \rightarrow \bigoplus_{j=1}^{\beta_{d-4}} \mathcal{O}_{\PP(\mathcal{E})}(-(d-3)H + b_j^{(d-4)}R)  \rightarrow   \\
\cdots \rightarrow \bigoplus_{j=1}^{\beta_1} \mathcal{O}_{\PP(\mathcal{E})}(-2H + b_j^{(1)}R) \rightarrow \mathcal{O}_{\PP(\mathcal{E})} \rightarrow \mathcal{O}_C \rightarrow 0,
\end{align*}
for certain integers $b_j^{(i)}$, called the \emph{splitting types}. These splitting types satisfy a duality statement~\cite[Cor.\,4.4]{schreyer}, for every $i\in \{1, \ldots, d-3\}$ we have that
\[
\{b_{j}^{(d-2-i)}\}_j = \{g+d-1-b^{(i)}_j\}_j.
\]
This can also be seen as a corollary of our duality result, Proposition~\ref{prop:duality}, in combination with Theorem~\ref{thm:schreyer.is.scrollar}.

The fibres of the map $\pi: \PP(\EE)\to \PP^1$ are isomorphic to $\PP^{d-2}$, and we equip these with homogeneous coordinates $x_1, \ldots, x_{d-1}$. We write $s,t$ for coordinates on $\PP^1$. This allows us to speak about defining equations for $C$, and more generally for the syzygies appearing in the Casnati--Ekedahl resolution. In more detail, the first map 
\[
\bigoplus_{j=1}^{\beta_1} \mathcal{O}_{\PP(\mathcal{E})}(-2H + b_j^{(1)}R) \rightarrow \mathcal{O_{\PP(\mathcal{E})}}
\]
is described by $\beta_1$ quadratic forms ($j=1, \ldots, \beta_1$)
\[
Q^j = \sum_{j_1+j_2+\ldots + j_{d_1} = 2} Q^{j}_{j_1, \ldots, j_{d-1}}(s,t) x_1^{j_1}\cdots x_{d-1}^{j_{d-1}},
\]
where $Q^j_{j_1, \ldots, j_{d-1}}(s,t)\in F[s,t]$ is homogeneous of degree 
\[
j_1e_1 + j_2e_2 + \ldots + j_{d-1}e_{d-1} - b_j^{(1)}.
\]
For $i\in \{1, \ldots, d-4\}$, the map
\[
\bigoplus_{j=1}^{\beta_{i+1}} \mathcal{O}_{\PP(\mathcal{E})}(-(i+2)H + b_j^{(i+1)}R) 
\to \bigoplus_{j=1}^{\beta_i} \mathcal{O}_{\PP(\mathcal{E})}(-(i+1)H + b_j^{(i)}R) 
\]
can be represented by a $\beta_i\times \beta_{i+1}$ matrix whose $(j_1,j_2)$ entry is a linear form $L_1(s,t)x_1 + \ldots L_{d-1}(s,t)x_{d-1}$ where $L_j(s,t)\in F[s,t]$ is homogeneous of degree 
\[
e_j + b_{j_1}^{(i)} - b_{j_2}^{(i+1)}.
\]
The last step in the resolution may be described similarly using quadratic forms.

To prove Theorem~\ref{thm:schreyer.is.scrollar} we apply the construction of the resolution from section \ref{sec:resolution} to the generic fibre of $C\to \PP^1$. By applying this construction with reduced bases, we find exactly the Casnati--Ekedahl resolution as above. 

\begin{proof}[Proof of Theorem~\ref{thm:schreyer.is.scrollar}]
The proof is identical to \cite{syzrep}, replacing the use of the maximal order of $L/k(t)$ by the $S_d$-closure of $\OO_K$. Crucial for this proof to work is the volume formula, Corollary~\ref{cor:volume.formula}.
\end{proof}

\section{Examples and applications}\label{sec:examples}

In this section we discuss some examples and applications of our work. We discuss a more general Maroni bound on the largest scrollar invariant of irreducible resolvent curves, which we apply to give new bounds on the splitting types. This extends the bounds from \cite{syzrep} beyond the simply branched case. We then give a detailed treatment of the most important resolvents in degree $d\leq 6$.

\subsection{The Maroni bound}

There are many results in the literature about understanding what the scrollar invariants of a cover $\varphi: C\to \PP^1$ can be, see e.g.\ \cite{bujokas_patel, coppensmartens, sameera, ohbuchi, peikertrosen, ballico, deopurkar_patel_bundles, landesmanlitt}. Using ideas from~\cite[Thm.\,3.1]{2torsionclassgroup}, we obtain the following general Maroni bound for resolvent curves.

\begin{theorem}[Maroni bound]\label{thm:maroni.bound}
Let $d\geq 4$ be an integer and $k$ a field of characteristic zero or larger than $d$. Let $\varphi: C\to \PP^1$ be a degree $d$ morphism over $k$, where $C$ is a nice curve of genus $g$. Let $H$ be a subgroup of $S_d$ for which $\res_H C$ is irreducible. Denote by $r_1 \leq \ldots \leq r_{[S_d:H]-1}$ the scrollar invariants of $\res_H \varphi: \res_H C\to \PP^1$. Then 
\[
r_{[S_d:H]-1} \leq \frac{2g(\res_H C) + 2[S_d:H] - 2}{[S_d:H]},
\]
where $g(\res_H C)$ is the arithmetic genus of $\res_H C$.
\end{theorem}

The proof of this result crucially depends on the fact that $\res_H C$ is irreducible.

\begin{proof}
Let $n=[S_d:H]$, which is the degree of the map $\res_H C\to \PP^1$. Denote by $K$ the function field of $C$ and by $L$ its $S_d$-closure, so that $L^H$ is the function field of $\res_H C$. Let $1, \alpha_1, \ldots, \alpha_{n-1}$ be a reduced basis for $L^H$, with scrollar invariants $r_1\leq \ldots \leq r_{n-1}$. Consider the $(n-2)\times (n-2)$ matrix $M = (m_{ij})$ where $m_{ij}$ is the coefficient of $\alpha_{n-1}$ in $\alpha_i\alpha_j$. We claim that $\det M \neq 0$. If not, then there is an element $\beta$ in the span of $\alpha_1, \ldots, \alpha_{n-2}$ such that multiplication by $\beta$ fixes the $k(t)$-subspace $A$ of $L^H$ spanned by $1, \alpha_1, \ldots, \alpha_{n-2}$. Since $L^H$ is a field, so is $k(t)(\beta)$. Hence $A$ is a $k(t)(\beta)$-vector subspace of $L^H$. But $k(t)(\beta)$ has $k(t)$-dimension at least $2$, while $A$ is of codimension $1$ in $L^H$. This is a contradiction.

It follows that for some permutation $\pi\in S_{n-2}$, $m_{i\pi(i)} \neq 0$ for all $i$. For each $i$, both $v_{n-1}$ and $v_iv_{\pi(i)}$ complete $\{1, v_1, \ldots, v_{n-2}\}$ to a $k(t)$-basis of $L^H$. Since we are working with a reduced basis we have that $r_{n-1}\leq r_i + r_{\pi(i)}$. Summing over all $i$ and adding $2r_{n-1}$ gives that 
\[
nr_{n-1} \leq \sum_{i=1}^{n-2}(r_i + r_{\pi(i)}) + 2r_{n-1} \leq 2\sum_{i=1}^{n-1} r_i = 2g(\res_H C) +2n - 2.\qedhere
\]
\end{proof}

Applying this to $H = S_\lambda$ we obtain the following. 

\begin{corollary}\label{cor:maroni.bound}
Let $\varphi: C\to \PP^1$ be a degree $d\geq 4$ morphism over a field $k$ of characteristic $0$ or larger than $d$. Let $\lambda = (d_1, \ldots, d_r)$ be a partition of $d$ and assume that $R_{S_\lambda}C$ is irreducible. Then the scrollar invariants $e_{\lambda, j}$ of $\lambda$ with respect to $\varphi$ satisfy
\[
e_{\lambda, j}\leq \frac{d^2 - \sum_i d_i^2}{d(d-1)}(g+d-1).
\]
\end{corollary}

\begin{proof}
Apply the previous theorem to $H=S_\lambda$ and use the fact that $V_\lambda$ appears in $\Ind_{S_\lambda}^{S_d} \mathbf{1}$.
\end{proof}

For the Schreyer invariants of a curve we obtain the following.

\begin{corollary}\label{cor:maroni.bound.for.schreyer}
Let $\varphi: C\to \PP^1$ be a degree $d$ morphism over $k$, whose associated Galois closure has Galois group $S_d$ or $A_d$. For $i \in \{1, 2, \ldots, d-3\}$, the elements $b_j^{(i)}$ of the splitting type of the $i$th syzygy bundle in the Casnati--Ekedahl resolution of $C$ with respect to $\varphi$ are contained in
\[
\left[ \tfrac{i(i+1)+2}{d(d-1)}(g+d-1)  , \tfrac{(i+1)(2d-i-2)-2}{d(d-1)}(g+d-1)  \right].
\]
In particular, all $b_j^{(i)}$ are non-negative.
\end{corollary}

\begin{proof}
This follows from Corollary~\ref{cor:maroni.bound}, Theorem~\ref{thm:schreyer.is.scrollar} and Proposition~\ref{prop:resolvent.irre}, or see~\cite[Cor.\,48]{syzrep} for more details.
\end{proof}

\subsection{Casnati's theorem}

Let $\varphi: C\to \PP^1$ be a degree $4$ cover. The most interesting resolvent is the one with respect to the order $8$ dihedral subgroup $D_4\subset S_4$ (e.g.\ generated by $(1234)$ and $(13)$). Classically, this resolvent was used by Lagrange to reduce solvability of degree $4$ polynomials to that of degree $3$~\cite{coxgalois}. Geometrically, this resolvent corresponds to the Recillas' trigonal construction~\cite{recillas}. The following result is a more precise version of a theorem due to Casnati~\cite{casnati}, see also~\cite{deopurkar_patel}.

\begin{corollary}[Casnati]
Let $C\to \PP^1$ be a degree $4$ cover in characteristic not $2$ or $3$. Let $K/k(t)$ be the corresponding function field extension, and let $M/k(t)$ be the Galois closure of $K/k(t)$ with Galois group $G\subset S_4$. Denote by $b_1, b_2$ the Schreyer invariants of $C\to \PP^1$. Then the scrollar invariants of $\res_ {D_4} C\to \PP^1$ are
\[
b_1, b_2.
\]
Moreover, $\res_{D_4} C$ is smooth if and only if $C\to \PP^1$ has no ramification of type $(2,2)$ or $(4)$, and $\res_{D_4} C$ is irreducible if and only if $G = S_4$ or $G = A_4$.

If $\res_{D_4}C$ is irreducible, then its arithmetic genus is $g+1$, and we have the bound $b_i\leq \frac{2}{3}g+2$.
\end{corollary}

\begin{proof}
The decomposition $\Ind_{D_4}^{S_4} \mathbf{1} = V_{(4)}\oplus V_{(2,2)}$ in combination with Theorem~\ref{thm:schreyer.is.scrollar} shows that the scrollar invariants of $\res_{D_4} C$ are indeed $b_1$ and $b_2$. The fact about smoothness and irreducibility follows from Propositions \ref{prop:resolvent.maximal} and \ref{prop:resolvent.irre}. The genus of $\res_{D_4} C$ may be computed using Corollary~\ref{cor:genus.resolvent}. The last part follows from Corollary~\ref{cor:maroni.bound.for.schreyer}.
\end{proof}

\begin{remark}
Note that in this setting, if $\res_{D_4} C$ is smooth then it is also geometrically irreducible. Indeed, if $C\to \PP^1$ is simply branched then $G=S_4$, implying that $\res_{D_4} C$ is geometrically irreducible. If there is $(3,1)$-ramification in $C\to \PP^1$ then $G$ can only be $S_4$ or $A_4$, since these are the only transitive subgroups of $S_4$ containing $3$-cycles.

One can also see this geometrically as follows. Given equations of $C$, it is possible to write down an explicit equation for $\res_{D_4} C$ inside its scroll $S$, which is a Hirzebruch surface, see~\cite[p.\,1351]{vangeemen} or~\cite[p.\,1340]{bhargavaquarticrings}. If $\res_{D_4} C$ were reducible, then a computation of the self intersection number of $\res_{D_4} C$ inside $S$ would contradict smoothness of $\res_{D_4}C$. We refer to~\cite[Sec.\,3]{kresch_curves_2001} for a similar argument.
\end{remark}

One might wonder about what the non-singular model of $\res_{D_4} C$ looks like. So let us denote by $\res_{D_4} \tilde{C}$ the non-singular model of $\res_{D_4} C$, and let us assume that $k$ is algebraically closed. In particular, we are interested in the ramification of the map $\res_{D_4} \tilde{C}\to \PP^1$, and the singularities of $\res_{D_4} C$. Both of these will only depend on the ramification of the original map $\varphi: C\to \PP^1$, and so we work locally. For this, we denote by $\hat{\OO}_L$ the completion of $\OO_L$ at $0$, and similarly $\hat{\tilde{\OO}}_L$ for the completion of the maximal order $\tilde{\OO}_L$ of $L$. Both of these are $k[[t]]$-algebras. Then the situation is summarized in the following addendum.

\begin{addendum}
With notation as above, the discriminants of $\hat{\OO}_L^{D_4}$, $\hat{\tilde{\OO}}_L^{D_4}$, and the ramification pattern $e'$ of $0$ in $\res_{D_4} \tilde{C}\to \PP^1$ are given by the following table, where $e$ is the ramification of $0$ in $\varphi: C\to \PP^1$.
\begin{center}
\begin{longtable}{ c | c c c c}
 $e$ & $\disc(\hat{\OO}_K)$ & $\disc(\hat{\OO}_L^{D_4})$ & $\disc(\hat{\tilde{\OO}}_L^{D_4})$ & $e'$ \\
$(1^4)$ & $1$ & $1$ & $1$ & $(1^3)$ \\
$(2, 1^2)$ & $t$ & $t$ & $t$ & $(2,1)$ \\
$(3,1)$ & $t^2$ & $t^2$ & $t^2$ & $(3)$ \\
$(4)$ & $t^3$ & $t^3$ & $t$ & $(2,1)$ \\
$(2,2)$ & $t^2$ & $t^2$ & $1$ & $(1^3)$
\end{longtable}
\end{center}
\end{addendum}

\begin{proof}
The discriminant of $\hat{\tilde{\OO}}_L^{D_4}$ can be computed by using Lemma~\ref{lem:disc.maximal.order.resolvent}, from which the result follows.
\end{proof}

Note that it is also possible to deduce the genus of $\res_{D_4} \tilde{C}$ from this table using the Riemann--Hurwitz formula. For example, if there is only $(2,2)$ ramification in $\varphi$, then $\res_{D_4}\tilde{C}$ is an unramified triple cover of $\PP^1$. This is only possible if $\res_{D_4}\tilde{C}$ is a disjoint union of three copies of $\PP^1$.

\subsection{Cayley's sextic resolvent}

Let $\varphi: C\to \PP^1$ be of degree $d=5$ now. The most important resolvent in this case is the one with respect to the group $\AGL_1(\FF_5)$ generated by $(12345)$ and $(1243)$ in $S_5$. This is a group of order $20$ and the corresponding resolvent is known as \emph{Cayley's sextic resolvent}. This resolvent can be used to determine whether a degree $5$ polynomial is solvable by radicals~\cite{coxgalois}. In Bhargava's work on quintic ring parametrizations~\cite{bhargavaquinticrings}, this resolvent also appears. A proof of the following theorem can likely also be extracted from that work.

\begin{corollary}
Let $\varphi: C\to \PP^1$ be a degree $5$ cover in characteristic not $2$ or $3$, where $C$ is a nice curve of genus $g$. Let $K/k(t)$ be the corresponding function field extension, and let $M/k(t)$ be the Galois closure of $K/k(t)$ with Galois group $G\subset S_5$. Denote by $b_1^{(2)}, b_2^{(2)}, \ldots, b_5^{(2)}$ the splitting types of the second syzygy module in the Casnati--Ekedahl resolution of $\varphi$. Then the scrollar invariants of $\res_{\AGL_1(\FF_5)} C\to \PP^1$ are
\[
b_1^{(2)}, b_2^{(2)}, \ldots, b_5^{(2)}.
\]
Moreover, $\res_{\AGL_1(\FF_5)} C$ is smooth if and only if $C\to \PP^1$ is simply branched, and $\res_{\AGL_1(\FF_5)} C$ is irreducible if and only if $G = S_5$ or $G = A_5$.

If $\res_{\AGL_1(\FF_5)} C$ is irreducible, then it has arithmetic genus $3g+7$ and we have the bound $b_5^{(2)} \leq \frac{4}{5}(g+4)$.
\end{corollary}

\begin{proof}
That the scrollar invariants are as described follows from the decomposition $\Ind_{\AGL_1(\FF_5)}^{S_5} \mathbf{1} = V_{(5)} \oplus V_{(2^2,1)}$, in combination with Theorem~\ref{thm:schreyer.is.scrollar} and Theorem~\ref{thm:scrollar.invs.resolvent}. The statements about smoothness and irreducibility follow from Propositions \ref{prop:resolvent.irre} and \ref{prop:resolvent.maximal}. The final bound follows from Theorem~\ref{thm:maroni.bound}.
\end{proof}

As for the possible ramification patterns in the non-singular model of the resolvent, we have the following. We use the same notation as in the previous subsection.

\begin{addendum}
With notation as above, the discriminant of $\hat{\OO}_L^{\AGL_1(\FF_5)}$, $\hat{\tilde{\OO}}_L^{\AGL_1(\FF_5)}$, and the ramification pattern $e'$ of $0$ in $\res_{\AGL_1(\FF_5)} \tilde{C}\to \PP^1$ is given by the following table, where $e$ is the ramification of $0$ in $\varphi: C\to \PP^1$.
\begin{center}
\begin{longtable}{ c | c c c c}
 $e$ & $\disc(\hat{\OO}_K)$ & $\disc(\hat{\OO}_L^{\AGL_1(\FF_5)})$ & $\disc(\hat{\tilde{\OO}}_L^{\AGL_1(\FF_5)})$ & $e'$ \\
$(1^5)$		&	$1$		&	$1$		&	$1$		&	$(1^6)$ \\
$(2,1^3)$	&	$t$		&	$t^3$	&	$t^3$	&	$(2,1^4)$ \\
$(3,1^2)$	&	$t^2$	&	$t^6$	&	$t^4$	&	$(3,3)$ \\
$(4,1)$		&	$t^3$	&	$t^9$	&	$t^3$	&	$(4, 1^2)$ \\
$(5)$		&	$t^4$	&	$t^{12}$&	$t^4$	&	$(5,1)$ \\
$(2,2,1)$	&	$t^2$	&	$t^6$	&	$t^2$	&	$(2^2,1^2)$ \\
$(3,2)$		&	$t^3$	&	$t^9$	&	$t^5$	&	$(6)$.
\end{longtable}
\end{center}
\end{addendum}

\begin{proof}
The discriminant of $\hat{\tilde{\OO}}_L^{D_4}$ can be computed by using Lemma~\ref{lem:disc.maximal.order.resolvent}. For the ramification behaviour $e'$ of $0$ in $\tilde{\OO}_L^H$ we note that since the Galois closure of $L^H$ in $L$ is equal to $L$, we must have that $\lcm(e')=\lcm(e)$. Combining this fact with the information about the discriminant allows us to complete the table as given.
\end{proof}

\subsection{Exotic resolvent}

From degree $6$ onwards, there are scrollar invariants for which we have no geometric interpretations. For $d=6$ these are the scrollar invariants with respect to the partitions $(3,3)$ and $(2,2,2)$. Note that these are dual to each other by Proposition~\ref{prop:duality}. Let $\varphi: C\to \PP^1$ be a degree $6$ cover, where $C$ is a nice curve of genus $g$ and denote by $a_1\leq \ldots \leq a_5$ the scrollar invariants of $\varphi$ with respect to the partition $(2^3)$. We will call the $a_i$ the \emph{exotic invariants of $C$}. By the volume formula, Corollary~\ref{cor:volume.formula}, we have that $a_1+\ldots + a_5 = 3g+15$. 

It is possible to realize these scrollar invariants via a resolvent curve as well. For this, consider the \emph{exotic} embedding $S_5\hookrightarrow S_6$ whose image $S_5'$ is a transitive subgroup of $S_6$. This subgroup is unique up to conjugation, one instance being generated by $(1234)$ and $(1562)$.

Define the following two transitive subgroups of $S_6$:
\[
F_{36} = \langle (246), (15)(24), (1452)(36)\rangle, \quad F_{72} = \langle (246), (24), (14)(25)(36)\rangle.
\]
These have orders $36$ and $72$ respectively.

\begin{corollary}
Let $\varphi: C\to \PP^1$ be a degree $6$ cover in characteristic not $2, 3$ or $5$, where $C$ is a nice curve of genus $g$. Let $K/k(t)$ be the corresponding function field extension, and let $M/k(t)$ be the Galois closure of $K/k(t)$ with Galois group $G\subset S_6$. Denote by $a_1, a_2, \ldots, a_5$ the exotic invariants of $\varphi$. Then the scrollar invariants of $\res_{S_5'} C\to \PP^1$ are
\[
a_1, \ldots, a_5.
\]
Moreover, $\res_{S_5'} C$ is smooth if and only if $C\to \PP^1$ is simply branched, and $\res_{S_5'} C$ is irreducible if and only if $G = S_6, A_6, F_{36}$ or $F_{72}$.

If $\res_{S_5'} C$ is irreducible, then it has arithmetic genus $3g+10$ and we have the bound $a_5 \leq \frac{4}{5}(g+5)$.
\end{corollary}

\begin{proof}
The corollary follows in exactly the same way as the previous results, using that $\Ind_{S_5'}^{S_6}\mathbf{1} = V_{(6)} \oplus V_{(2^3)}$.
\end{proof}

Again, it is possible to analyse the smooth model of the resolvent $\res_{S_5}' \tilde{C}$.

\begin{addendum}
The discriminant of $\hat{\OO}_L^{S_5'}$, $\hat{\tilde{\OO}}_L^{S_5'}$, and the ramification pattern $e'$ of $0$ in $\res_{\AGL_1(\FF_5)} \tilde{C}\to \PP^1$ is given by the following table, where $e$ is the ramification of $0$ in $\varphi: C\to \PP^1$.
\begin{center}
\begin{longtable}{ c | c c c c}
 $e$ & $\disc(\hat{\OO}_K)$ & $\disc(\hat{\OO}_L^{S_5'})$ & $\disc(\hat{\tilde{\OO}}_L^{S_5'})$ & $e'$ \\
$(1^6)$			&	$1$			&	$1$			&	$1$			&	$(1^6)$ \\
$(2, 1^4)$		&	$t$			&	$t^3$		&	$t^3$		&	$(2^3)$ \\
$(2^2, 1^2)$		&	$t^2$		&	$t^6$		&	$t^2$		&	$(2^2, 1^2)$ \\		
$(2^3)$			&	$t^3$		&	$t^9$		&	$t$			&	$(2, 1^4)$ \\
$(3, 1^3)$		&	$t^2$		&	$t^6$		&	$t^4$		&	$(3^2)$ \\
$(3, 2, 1)$		&	$t^3$		&	$t^9$		&	$t^5$		&	$(6)$ \\
$(3^2)$			&	$t^4$		&	$t^{12}$		&	$t^2$		&	$(3, 1^3)$ \\
$(4, 1^2)$		&	$t^3$		&	$t^{12}$		&	$t^3$		&	$(4, 1^2)$ \\
$(4, 2)$			&	$t^4$		&	$t^{12}$		&	$t^4$		&	$(4,2)$ \\
$(5,1)$			&	$t^4$		&	$t^{12}$		&	$t^4$		&	$(5,1)$ \\
$(6)$			&	$t^5$		&	$t^{15}$		&	$t^3$		&	$(3,2,1)$ \\
\end{longtable}
\end{center}
\end{addendum}
\begin{proof}
The techniques for the proof are similar to the previous addendum.
\end{proof}

\begin{remark}
The map sending a degree $6$ cover $C\to \PP^1$ to the degree $6$ cover $\res_{S_5'} \tilde{C}$ is an involution. In other words, applying this construction twice gives back the original curve. This is also visible in the ramification patterns in the table above.
\end{remark}

\bibliographystyle{amsplain}
\bibliography{MyLibrary}

\end{document}